\documentclass[11pt,a4paper,oneside]{amsart}

\newcounter{commentcounter}

\usepackage{amsmath} 
\usepackage{amssymb} 
\usepackage{amsthm} 
\usepackage{stmaryrd} 
\usepackage[english]{babel} 
\usepackage[font=small,justification=centering]{caption} 
\usepackage[nodayofweek]{datetime}
\usepackage[shortlabels]{enumitem} 
\usepackage[T1]{fontenc} 
\usepackage[utf8]{inputenc} 
\usepackage{ifthen} 
\usepackage{mathabx} 
\usepackage{mathtools} 
\usepackage[dvipsnames]{xcolor} 
\usepackage[pdftex,  colorlinks=true, pagebackref=true]{hyperref} 
    \hypersetup{urlcolor=RoyalBlue, linkcolor=RoyalBlue,  citecolor=black}
\usepackage{setspace} 
\usepackage{tikz-cd} 
\usepackage{xfrac} 
\usepackage[capitalize]{cleveref} 

\renewcommand*{\backref}[1]{}
\renewcommand*{\backrefalt}[4]
{
    \ifcase #1
        No citation in the text.
    \or
        Cited on Page #2.
    \else
        Cited on Pages #2.
    \fi
}

\makeatletter
\@namedef{subjclassname@1991}{Mathematical subject classification 1991}
\@namedef{subjclassname@2000}{Mathematical subject classification 2000}
\@namedef{subjclassname@2010}{Mathematical subject classification 2010}
\@namedef{subjclassname@2020}{Mathematical subject classification 2020}
\makeatother

\newtheorem{thm}{Theorem}[section]
\newtheorem{lemma}[thm]{Lemma}
\newtheorem{corollary}[thm]{Corollary}
\newtheorem{prop}[thm]{Proposition}
\newtheorem{conjecture}[thm]{Conjecture}

\newtheorem{question}[thm]{Question}

\newtheorem{thmx}{Theorem}
\newtheorem{corx}[thmx]{Corollary}

\theoremstyle{definition}
\newtheorem{defn}[thm]{Definition}
\newtheorem{remark}[thm]{Remark}
\newtheorem*{remark*}{Remark}

\theoremstyle{plain}

\newtheoremstyle{TheoremNum}
        {8.0pt plus 2.0pt minus 4.0pt}{8.0pt plus 2.0pt minus 4.0pt} 
        {\itshape} 
        {-0.15cm} 
        {\bfseries} 
        {.} 
        { }  
        {\thmname{#1}\thmnote{ \bfseries #3}}
    \theoremstyle{TheoremNum}
    \newtheorem{duplicate}{}

\newcommand*{\claimproofname}{My proof}

\usepackage{thmtools,thm-restate} 

\DeclareMathOperator{\Aut}{\mathrm{Aut}}
\DeclareMathOperator{\Out}{\mathrm{Out}}
\DeclareMathOperator{\Inn}{\mathrm{Inn}}

\DeclareMathOperator{\Fix}{\mathrm{Fix}}
\DeclareMathOperator{\Stab}{\mathrm{Stab}}

\DeclareMathOperator{\MCG}{\mathrm{MCG}}

\DeclareMathOperator{\FJC}{\mathbf{FJC}}
\DeclareMathOperator{\FJCX}{\mathbf{FJC}_\mathbf{X}}
\DeclareMathOperator{\AC}{\mathcal{AC}}
\DeclareMathOperator{\VNil}{\mathbf{VNil}}
\DeclareMathOperator{\Per}{Per}

\newcommand{\cala}{{\mathcal{A}}}

\newcommand{\cale}{{\mathcal{E}}}
\newcommand{\calf}{{\mathcal{F}}}

\newcommand{\calh}{{\mathcal{H}}}

\newcommand{\calp}{{\mathcal{P}}}

\newcommand{\calt}{{\mathcal{T}}}

\newcommand{\VC}{\mathcal{VC}}

\newcommand{\GL}{\mathrm{GL}}

\newcommand{\CAT}{\mathrm{CAT}}

\DeclareMathOperator{\id}{id}

\DeclareMathOperator{\ad}{ad} 

\def\Z{\mathbb{Z}}

\newcommand{\NN}{\mathbb{N}}
\newcommand{\ZZ}{\mathbb{Z}}

\newcommand{\RR}{\mathbb{R}}

\newcommand{\Tcan}{T^{\mathrm{can}}}
\newcommand{\Tper}{T^{\mathrm{Per}}}
\newcommand{\Tref}{T^{\mathrm{ref}}}
\newcommand{\per}{\mathrm{Per}_{\mathrm{NP}}(\phi)}

\usepackage{tikz}
\usetikzlibrary{arrows,quotes}
\tikzstyle{blackNode}=[fill=black, draw=black, shape=circle]

\title[Automorphisms and the Farrell--Jones Conjecture]{Automorphisms of relatively hyperbolic groups and the Farrell--Jones Conjecture}

\author{Naomi Andrew}
\address[Current]{Naomi Andrew, Laboratoire de Math\'ematiques d'Orsay, Universit\'e Paris-Saclay, Batiment 307, F-91405 Orsay Cedex, France}
\author{Yassine Guerch}
\address[Current]{Yassine Guerch, Université de Caen Normandie, CNRS, Laboratoire de Mathématiques Nicolas Oresme,
14032 Caen Cedex 5, France}
\author{Sam Hughes}
\address[Current]{Sam Hughes, Rheinische Friedrich-Wilhelms-Universit\"at Bonn, Mathematical Institute, Endenicher Allee 60, 53115 Bonn, Germany}
\address[Former]{Naomi Andrew and Sam Hughes, Mathematical Institute, Andrew Wiles Building, Observatory Quarter, University of Oxford, Oxford OX2 6GG, United Kingdom}
\address[Former]{Yassine Guerch, Univ. Lyon, ENS de Lyon, UMPA UMR 5669, 46 allée d'Italie,
F-69364 Lyon cedex 07, France}
\email{naomi.andrew@universite-paris-saclay.fr}
\email{yassine.guerch@cnrs.fr}
\email{sam.hughes.maths@gmail.com}
\email{hughes@math.uni-bonn.de}
\date{\today}


\begin{document}
\begin{abstract}
We prove the fibred Farrell--Jones Conjecture (FJC) in $A$-, $K$-, and $L$-theory for a large class of suspensions of relatively hyperbolic groups, as well as for all suspensions of one-ended hyperbolic groups.  We deduce two applications: 

(1) FJC for the automorphism group of a one-ended group hyperbolic relative to virtually polycyclic subgroups; 

(2) FJC is closed under extensions of FJC groups with kernel in a large class of relatively hyperbolic groups.\\
Along the way we prove a number of results about JSJ decompositions of relatively hyperbolic groups which may be of independent interest.
\end{abstract}
\maketitle

\section{Introduction}

Let $G$ be a group.  The Farrell--Jones Conjecture (FJC) is one of the most prominent open conjectures in algebraic and differential topology.  In its simplest form the $K$-theoretic conjecture predicts that a certain assembly map
\[ H_n^G(\mathrm{pr})\colon H_n^G(\underline{\underline{E}}G;\mathbf{K}_R)\to K_n(RG) \]
is an isomorphism.  Here $\underline{\underline{E}}G$ is the classifying space for the family of virtually cyclic subgroups, $\mathbf{K}_R$ is the algebraic $K$-theory spectrum for the ring $R$, and $K_n(RG)$ is the algebraic $K$-theory of the group ring $RG$.  There are variants of the conjecture for Waldhausen's $A$-theory and for $L$-theory.  The conjecture for $L$-theory, as well as a detailed account of the Farrell--Jones Conjecture, and the objects involved can be found in W.~L{\"u}ck's book project \cite{LuckBookProj}.  For recent progress on $A$-theory the reader should consult \cite{EnkelmannLuckPieperUllmannWinges}.

Computing the algebraic $K$-theory of a group ring $RG$ is a very difficult problem.  In principle, knowing that FJC holds for $G$ gives a method of computing $K_n(RG)$ using equivariant algebraic topology. It also has a number of other applications, for example, to the Borel Conjecture \cite{bartels2012borel} and to computing the Whitehead group $\mathrm{Wh}(G)$. Knowledge of $\mathrm{Wh}(G)$ is a fundamental step in classifications of higher dimensional manifolds with fundamental group $G$.

From this point onward, by the Farrell--Jones Conjecture for $G$, we mean the most general setting, that is, the \emph{fibred Farrell--Jones Conjecture with respect to the family of virtually cyclic subgroups $\VC$}.  See for example \cite{EnkelmannLuckPieperUllmannWinges} and \cite{LuckBookProj} for a discussion of these terms.  We will denote the classes of groups satisfying the FJC for $X$-theory by $\FJCX $ where $X$ is $A$-, $K$-, or $L$-theory.

The class $\FJC_\mathbf{K}$ of groups satisfying FJC for algebraic $K$-theory is large: containing hyperbolic groups \cite{BartelsLuckReich2008}, many relatively hyperbolic groups \cite{bartels2017coarse}, $\CAT(0)$ groups \cite{Wegner2012} (see also \cite{bartels2012borel} and \cite{KasprowskiRuping2017}), soluble groups \cite{wegner2015farrell}, $\GL_n(\Z)$ \cite{BartelsLuckReichRuping2014} and more generally lattices in connected Lie groups \cite{BartelsFarrellLuck2014} and $S$-arithmetic groups \cite{Ruping2-16}, as well as mapping class groups \cite{BartelsBestvina2019Farrell}, normally poly-free groups \cite{BruckKielakWu2021}, and suspensions of virtually torsion free hyperbolic groups \cite{bestvina2023farrell}. The class enjoys many closure properties: it passes to arbitrary subgroups, finite index overgroups, and directed colimits.  For more information the reader is referred to the surveys \cite{BartelsLuckReich2008survey,LuckReich2005,Luck2010ICM,Bartels2016survey}.

One property that is not known is whether $\FJCX $ is closed under extensions $1 \to N \to \Gamma \to Q \to 1$. One direction of interest is to put conditions on $N$ so that $\Gamma$ is in $\FJCX $ whenever $Q$ is. By \cite[Theorem 2.7]{BartelsFarrellLuck2014} and \cite[Theorem~1.1(ii)]{EnkelmannLuckPieperUllmannWinges} this reduces to understanding cyclic extensions of $G$; which is to say the suspensions $N_\Phi=N \rtimes_\Phi \ZZ$, where $\Phi$ is some automorphism of $N$ defining this \emph{suspension}.

Intuitively, a group $G$ is \emph{hyperbolic relative to $\calp$} if its geometry is hyperbolic ``away from the subgroups $P \in \calp$.'' One (of many: see \cite{hruska2010relative} for the definitions as well as proofs of their equivalence) way to formalise this uses the notion of coning off a Cayley graph: take a vertex for every coset $gP$ of each element of $P$, and add an edge from each element of $gP$ to the new vertex. The group $G$ is hyperbolic relative to $\calp$ if the resulting graph is $\delta$-hyperbolic in the sense of Gromov, and \emph{fine}: every edge is contained in finitely many cycles of a given finite length. An automorphism of $G$ lies in the subgroup $\Aut(G,\calp)$ if it preserves the conjugacy classes of every subgroup $P \in \calp$. For more information on $\Aut(G,\calp)$ see \cite{minasyan2012fixed} and \cite{GuirardelLevitt2015}.

Recently, Bestvina, Fujiwara and Wigglesworth \cite{bestvina2023farrell} proved the suspension of a virtually torsion free hyperbolic group satisfies the Farrell--Jones conjecture.  We extend this result to a large class of relatively hyperbolic groups. 

\begin{duplicate}[\Cref{thmx:main}]
    Let $(G,\calp)$ be a virtually torsion-free or one-ended relatively hyperbolic group with $\calp$ finite and let $\Phi\in\Aut(G,\calp)$.  If for every $[P]\in \calp$ we have $P_\Phi \in \FJCX$, then $G_\Phi \in \FJCX$. 
\end{duplicate}

These hypotheses include, for instance, all suspensions of toral relatively hyperbolic groups and more generally one-ended or virtually torsion-free groups that are hyperbolic relative to virtually polycyclic or soluble subgroups. Note that this removes the assumption of virtual torsion-freeness in \cite{bestvina2023farrell} for one-ended hyperbolic groups.  This is pertinent since it is a well known question of Gromov whether every hyperbolic group is residually finite (and hence virtually torsion-free).

With infinitely ended groups more care is needed, we discuss this further in \Cref{sec.intro.oo_ends}.

\subsection{Applications}

Our first application is a result on extensions with relatively hyperbolic kernel.  A group is \emph{non-relatively hyperbolic} or NRH if it is not hyperbolic relative to a collection of proper subgroups.

\begin{duplicate}[\Cref{Corx:Extensions}]
    Let $(N,\calp)$ be a virtually torsion-free or one-ended relatively hyperbolic group such that $\calp$ consists of finitely many conjugacy classes of groups which are NRH and whose suspensions $P \rtimes_\Psi \ZZ$ are in $\FJCX$ for all automorphisms $\Psi$ of $P$.  Let $1\to N\to \Gamma\to Q\to 1 $ be a short exact sequence.  If $Q$ is in $\FJCX$, then $\Gamma$ is in $\FJCX$.
\end{duplicate}

The assumption that peripheral subgroups are NRH is needed for \cref{Corx:Extensions}, since it requires \cref{thmx:main} to hold for arbitrary automorphisms.  The key point being that $\Aut(N;\calp)$ has finite index in $\Aut(N)$ under this extra hypothesis.

It is a major open problem whether $\Out(F_N)$ satisfies $\FJCX$. Whilst we do not solve this, using \Cref{thmx:main} we are able to show automorphism groups of one-ended groups hyperbolic relative to virtually polycyclic groups satisfy $\FJCX$.  In particular, $\Aut(G)$ and $\Out(G)$ for $G$ a one-ended hyperbolic group satisfy $\FJCX$.

\begin{duplicate}[\Cref{thmx:Aut_FJC}]
    If $G$ is a one-ended group hyperbolic relative to finitely many conjugacy classes of virtually polycyclic groups, then $\Aut(G)$ and $\Out(G)$ are in $\FJCX$.
\end{duplicate}

\subsection{Remarks on the proofs}\label{sec.intro.oo_ends}
As is usual for (relatively) hyperbolic groups, there are two main flavours to our arguments, depending on the number of ends of $G$. In both cases we apply a result of Knopf \cite{knopf2019acylindrical} allowing us to deduce that a group acting acylindrically on a tree satisfies the Farrell--Jones conjecture if and only if its vertex groups do, though the source of the trees is different in each case.

For one-ended relatively hyperbolic groups, we have access to the powerful machinery of JSJ decompositions developed (in this generality) by Guirardel and Levitt \cite{guirardellevitt2017jsj}. We consider three related trees: the canonical JSJ decomposition $\Tcan$ relative to the peripheral subgroups $\calp$, a refinement $T^\phi$ of $\Tcan$ which better suits the study of an outer automorphism $\phi$ and another tree that we call $\Tper$. This tree is the canonical JSJ tree relative to the (non-elementary) periodic subgroups of the \emph{outer} automorphism $\phi$. That is, we require that the periodic subgroups of every representative $\Phi\ad_g$, are elliptic. Our main structural result about this tree is \cref{Coro:rigidvertexperiodic}: even without assuming that the periodic subgroups are finitely generated, they agree exactly with the rigid vertices of $\Tper$. We prove that this ensures that the induced action of $G \rtimes_{\Phi} \Z$ is acylindrical, and then analyse vertex groups that can appear in this new action.

We remark that the strong uniqueness properties of the JSJ decomposition imply that for a one-ended, torsion free hyperbolic group, the rigid vertex groups of the tree considered in \cite{bestvina2023farrell} agree with those in our $\Tper$.

We consider the case when $G$ is infinitely ended and has a finite index subgroup which is a free product of one-ended groups.  Being virtually torsion free is sufficient but not necessary for this to occur. 
Once we have a free product splitting we have the following combination-type theorem.

\begin{duplicate}[\Cref{Thm:combinationthm}]
        Let $G=G_1 \ast \ldots \ast G_k \ast F_N$ be a free product of finitely generated groups, let $\calf'=\{[G_1],\ldots,[G_k]\}$ and let $ \Phi \in \Aut(G,\calf')$. For every $i \in \{1,\ldots,k\}$, denote by $\Phi_i$ an element of the outer class of $\Phi$ preserving $G_i$. If for every $i \in \{1,\ldots,k\}$ the group $G_i\rtimes_{\Phi_i} \ZZ$ is in $\FJCX$, then $G \rtimes_\Phi \ZZ$ is in $\FJCX$.
\end{duplicate}


\Cref{Thm:combinationthm} is proved by induction on the Grushko rank $k+N$. There are two kinds of induction step, depending on whether the maximal periodic free factor system is \emph{sporadic} or not.  A free factor system $(G, \calf)$ is sporadic if $G \cong G_1 \ast G_2$ or $G \cong G_1 \ast \Z$. This division might seem unusual to experts; a more standard division (for instance, in \cite{bestvina2023farrell} as well as throughout the $\Out(F_n)$ literature) depends instead on if the automorphism is polynomially or exponentially growing. 
Polynomially growing automorphisms are always sporadic in this sense, but so are some exponentially growing automorphisms. The non-sporadic case uses Dahmani and Li's work on relative hyperbolicity for suspensions of free factors \cite{dahmani2022relative}, whereas in the sporadic case we use the fact that these splittings are \emph{rigid}.

These rigidity arguments hold equally well for sporadic Stallings--Dunwoody decompositions, and so we are still able to obtain some results without first passing to a finite index free product: see \Cref{prop:sporadic_stallings_dunwoody}.

\subsection{Fixed and Periodic Subgroups, the classes \texorpdfstring{$\AC(\VNil)$}{AC(VNIL)} versus \texorpdfstring{$\FJCX$}{FJC}, and localising invariants}
Some previous results of this flavour have concluded the stronger property that the suspension is in the class $\AC(\VNil)$. This class consists of groups having particularly nice actions on compact Euclidean retracts, see~\cite{BartelsBestvina2019Farrell} for a precise definition. Every group in this class satisfies the Farrell--Jones conjecture \cite{BartelsBestvina2019Farrell}. This class has similar closure properties to the class of groups satisfying the Farrell--Jones conjecture, except that $\FJCX$ is closed under directed colimits while $\AC(\VNil)$ is not known to be. However, for the majority of the paper we work directly with the class $\FJCX$.

The reason for this is that we have to understand the periodic subgroups of certain automorphisms as an ascending union of fixed subgroups, and consider the action of the automorphism on this subgroup. In general our hypotheses do not guarantee that this union stabilises --- we do not have a \emph{virtual neatness} property to rely on.
        
However, there are hypotheses that ensure virtual neatness, and if we assume these then again the suspensions will be in $\AC(\VNil)$. One set of sufficient conditions is,

\begin{duplicate}[\Cref{Coro:Slenderperiodicisfixedoneended}]
    Let $G$ be a hyperbolic group relative to a collection $\calp$ of slender groups and let $\Phi \in \Aut(G)$. There exists $N \in \NN^*$ such that $\mathrm{Per}(\Phi)=\Fix(\Phi^N)$ and $\mathrm{Per}(\Phi)$ is finitely generated.
\end{duplicate}     

If we add these hypotheses to our main theorem, we can prove the suspensions lie in $\AC(\VNil)$.
     
\begin{duplicate}[\Cref{thmx:acvnil}]
    Suppose $(G, \calp)$ is one-ended or virtually torsion free, and hyperbolic relative to finitely many conjugacy classes of virtually polycyclic subgroups. Then for every automorphism $\Phi$ of $G$, $\Gamma \coloneqq G \rtimes_{\Phi} \Z$ is in $\AC(\VNil)$.
\end{duplicate}

Except for replacing each periodic subgroup with the fixed subgroup of a power, the proof of this theorem is identical to the proof of \cref{thmx:main}. We discuss this in a little more detail after completing that proof.

\begin{remark*}
    Following work of Bunke, Kaprowski, and Winges \cite{bunke2021farrelljones} our results apply equally well to the Farrell--Jones Conjecture for localising invariants, that is, with coefficients in $H\colon \mathbf{Cat}_{\infty,\ast}^{\mathrm{Lex}}\to \mathbf{M}$ a lax monoidal finitary localising invariant with values in a stably monoidal and cocomplete stable $\infty$-category which admits countable products.  We refer the reader to the introduction of loc. cit. for more information.
\end{remark*}



\subsection{Structure of the paper}
\Cref{sec:Farrell--Jones} introduces the relevant background results on the Farrell--Jones conjecture. 

\Cref{sec:free_prod_background} contains definitions and results on free products and their automorphisms, needed for \cref{sec:infinite_ended}.
 
 \Cref{sec:JSJ} collects results on JSJ decompositions of one-ended relatively hyperbolic groups, and provides a lemma on acylindricity when passing to the action of a suspension. 
 
 The one-ended case of \Cref{thmx:main} is proved in \cref{sec:periodicJSJ} by careful analysis of a certain JSJ tree. From this analysis, we also deduce~\Cref{Coro:Slenderperiodicisfixedoneended}. 
 
 In \cref{sec:infinite_ended} we prove \cref{Thm:combinationthm} and the infinitely-ended case of \Cref{thmx:main}. Using these results we prove \Cref{thmx:main} and \Cref{thmx:acvnil}.
 
 In \cref{sec:applications}, we deduce \Cref{thmx:Aut_FJC} from \Cref{thmx:main}. 
 
 Finally, in \cref{sec:torsion} we extend \Cref{thmx:main} as far as possible with our current techniques to groups which are infinitely ended but do not split as free products. The arguments of the last three sections are almost independent of \cref{sec:JSJ,sec:periodicJSJ}, apart from requiring the background information on trees of cylinders from \cref{subsec:trees_of_cylinders}. 


\subsection*{Acknowledgements}
This work has received funding from the European Research Council (ERC) under the European Union's Horizon 2020 research and innovation programme (Grant agreement No. 850930). The second author was supported by the LABEX MILYON of Université de Lyon. The authors would like to thank Damien Gaboriau, Dominik Kirstein, Gilbert Levitt, and Ric Wade for helpful conversations.  The authors would also like to thank the anonymous referee for a number of helpful comments and clarifications and for suggesting an elegant simplification of some of our arguments.

\section{Background on the Farrell--Jones conjecture}\label{sec:Farrell--Jones}

For full context and background on the Farrell--Jones conjecture, see for instance L\"uck's book project \cite{LuckBookProj}. In this section, we recall some properties of the class $\FJCX $ of groups which satisfy the Farrell--Jones conjecture for $X$-theory where $X$ is $A$, $K$, or $L$.

\begin{thm}\label{Thm:FJCclosedsubgroups}
    The class $\FJCX $ is closed under the following operations:
    \begin{enumerate}
        \item taking subgroups;
        \item taking finite index overgroups;
        \item finite direct products;
        \item finite free products;
        \item directed colimits.
    \end{enumerate}
\end{thm}
\begin{proof}
    The cases of $K$- and $L$-theory are given in \cite[Theorem~2.1]{gandini2015farrell}.  The case of $A$-theory is \cite[Theorem~1.1(ii)]{EnkelmannLuckPieperUllmannWinges}.
\end{proof}

While $\FJCX$ is not known to be closed under extensions, there is a partial result which we make use of.

\begin{thm}\label{thm:extensions}
Let $1\to N \to \Gamma \to Q \to 1$ be a short exact sequence with $N\in \FJCX$. If for every infinite cyclic subgroup $C$ of $Q$, the preimage of $C$ in $\Gamma$ belongs to $\FJCX$, then $\Gamma$ belongs to $\FJCX$.
\end{thm}
\begin{proof}
    The cases of $K$- and $L$-theory are given in \cite[Theorem 1.7]{BartelsFarrellLuck2014}.  The case of $A$-theory is given in \cite[Theorem 1.1(ii)]{EnkelmannLuckPieperUllmannWinges}.
\end{proof}

Here is an easy, mild strengthening of commensurability towards virtual isomorphism.

\begin{lemma}
    \label{lem:finite_extensions}
    Let $1\to N \to \Gamma \to Q \to 1$ be a short exact sequence with $N$ finite. Then if $Q$ is in $\FJCX$ then so is $\Gamma$, and if $\Gamma$ is residually finite and in $\FJCX$ then so is $Q$.
\end{lemma}
\begin{proof}
    For the first statement apply \cref{thm:extensions} to the short exact sequence, noting that both finite groups and virtually cyclic groups are in $\FJCX$. For the second, observe that if $\Gamma$ is residually finite then there is a finite index subgroup $\Gamma_0$ of $\Gamma$ whose intersection with $N$ is trivial, and then $\Gamma_0 \cong Q_0$ for some finite index subgroup $Q_0$ of $Q$. The result follows from commensurability.
\end{proof}

We refer for instance to the work of Bowditch~\cite{bowditch2012relatively} for the definition of a relatively hyperbolic group.

\begin{thm}[Bartels]\label{Thm:FJCrelativelyhyp}
    Let $G$ be a group hyperbolic relative to a collection $\{[P_1],\ldots,[P_n]\}$ of conjugacy classes of subgroups. If, for every $i \in \{1,\ldots,n\}$, we have $P_i \in \FJCX $, then $G \in \FJCX $.
\end{thm}
\begin{proof}
    This result is due to Bartels.  The cases of $K$- and $L$-theory are \cite[Corollary~4.6]{bartels2017coarse}.  The case of $A$-theory is also ostensibly due to Bartels combined with some recent developments on the $A$-theoretic FJC.  We sketch the relevant details.  The key here is that Bartels' space $\Delta$ for a relatively hyperbolic group pair $(G,\calp)$ is finitely $\calp$-amenable (see \cite[Theorem~3.1]{bartels2017coarse}).  By \cite[Proof of Theorem~1.8(a)]{knopf2019acylindrical} this implies that $G$ is strongly transfer reducible over $\calf$.  The result now follows from \cite[Theorem~6.19]{EnkelmannLuckPieperUllmannWinges}.
\end{proof}

Let $G$ be a group acting by isometries on a tree $T$. Recall that the action is \emph{acylindrical} if there exists $K\geq 0$ such that the stabiliser of any geodesic path of length at least $K$ is of bounded size.

\begin{thm}[Knopf]\label{Thm:FJCacylindrical}
   Let $G$ be a group acting acylindrically by isometries on a tree $T$. If every vertex stabiliser belongs to $\FJCX $, then $G$ belongs to $\FJCX $.
\end{thm}
\begin{proof}
    The result is due to S.~Knopf.  For $K$-theory we refer to \cite[Corollary~4.2]{knopf2019acylindrical}. The result for $L$-theory is \cite[Corollary~4.3]{knopf2019acylindrical}, note that here one has the additional hypothesis that index $2$ overgroups of the stabilisers in $G$ must satisfy $\FJC_\mathbf{L}$.  But this follows from \Cref{Thm:FJCclosedsubgroups}.  For $A$-theory, as in Bartels' result, one combines finite $\calf$-amenability \cite[Proposition~4.1]{knopf2019acylindrical} with the recent developments for $A$-theory \cite[Proof of Theorem~1.8(a)]{knopf2019acylindrical} and \cite[Theorem~6.19]{EnkelmannLuckPieperUllmannWinges}.
\end{proof}

Let $G$ be a group and let $\Phi \in \Aut(G)$. Let $\mathrm{Per}(\Phi)=\langle \Fix(\Phi^n) \rangle_{n \ge 1}$ be the periodic subgroup of $\Phi$. At several points in our arguments we will need the following lemma.

\begin{lemma}\label{Lem:PerFJC}
    Let $G$ be a group belonging to $\FJCX$ and let $\Phi \in \Aut(G)$. The group $\mathrm{Per}(\Phi) \rtimes_\Phi \ZZ$ belongs to $\FJCX$.
\end{lemma}

\begin{proof}
     Note that $$\mathrm{Per}(\Phi) \rtimes_\Phi \ZZ=\langle \Fix(\Phi^n) \rangle_{n \ge 1} \rtimes_\Phi \ZZ=\bigcup_{n\ge 1} \left(\Fix(\Phi^{n!}) \rtimes_\Phi \ZZ\right),$$ which is an increasing union of subgroups. Therefore, by Theorem~\ref{Thm:FJCclosedsubgroups}~$(5)$, it suffices to prove that, for every $n\in \NN$, the group $\Fix(\Phi^{n!}) \rtimes_\Phi \ZZ$ belongs to $\FJCX$. 

     Let $n\in \NN$. Note that $\Fix(\Phi^{n!}) \rtimes_{\Phi^{n!}} \ZZ$ is a finite index subgroup of $\Fix(\Phi^{n!}) \rtimes_\Phi \ZZ$. By Theorem~\ref{Thm:FJCclosedsubgroups}~$(2)$, the group $\Fix(\Phi^{n!}) \rtimes_\Phi \ZZ$ belongs to $\FJCX$ if and only if the group $\Fix(\Phi^{n!}) \rtimes_{\Phi^{n!}} \ZZ$ belongs to $\FJCX$.

    The group $\Fix(\Phi^{n!}) \rtimes_{\Phi^{n!}} \ZZ$ is isomorphic to $\Fix(\Phi^{n!}) \times \ZZ$. By Theorem~\ref{Thm:FJCclosedsubgroups}, the group $\ZZ$ belongs to $\FJCX$. Since $G \in \FJCX$ and since $\FJCX$ is closed under taking subgroups by Theorem~\ref{Thm:FJCclosedsubgroups}~$(1)$, the group $\Fix(\Phi^{n!})$ belongs to $\FJCX$. Since $\FJCX$ is closed under taking direct products by Theorem~\ref{Thm:FJCclosedsubgroups}~$(3)$, the group $\Fix(\Phi^{n!}) \times \ZZ$ and hence the group $\mathrm{Per}(\Phi) \rtimes_\Phi \ZZ$ belongs to $\FJCX$.
\end{proof}

\section{Free Products of Groups and their Automorphisms}\label{sec:free_prod_background}
\subsection{Free products of groups}

Let $N \in \NN$, let $G_1,\ldots,G_k$ be countable groups and let $G=G_1 \ast \ldots \ast G_k \ast F_N$. Let $\calf=\{[G_1],\ldots,[G_k]\}$ be the set consisting of the conjugacy classes of the $G_i$. We refer to $(G,\calf)$ as a \emph{free product}. 

An element $g \in G$ is \emph{peripheral} if there exists $[A]\in \calf$ with $g\in A$. Otherwise, $g$ is \emph{nonperipheral}. A subgroup $P$ of $G$ is \emph{peripheral} if every element of $P$ is peripheral, and is \emph{nonperipheral} otherwise.

A \emph{free factor system of $(G,\calf)$} is a set $\calf'=\{[A_1],\ldots,[A_{\ell}]\}$ of conjugacy classes of proper subgroups of $G$ such that:
\begin{enumerate}
    \item for every $i \in \{1,\ldots,k\}$, there exists $[A]\in \calf'$ such that $G_i \subseteq A$;
    \item there exists a subgroup $B$ of $G$ such that $G=A_1 \ast \ldots \ast A_{\ell} \ast B$.
\end{enumerate}

The set of free factor systems of $G$ is equipped with a partial order where $\calf_1 \leq \calf_2$ if, for every $[A_1]\in \calf_1$, there exists $[A_2] \in \calf_2$ with $A_1 \subseteq A_2$. A free factor system $\calf'$ is \emph{sporadic} if either $\calf'=\{[A_1],[A_2]\}$ and $G=A_1\ast A_2$ or $\calf'=\{[A_1]\}$ and $G=A_1 \ast \ZZ$. Otherwise, the free factor system $\calf'$ is \emph{nonsporadic}. The free product $(G,\calf)$ is \emph{sporadic} (resp. \emph{nonsporadic}) if $\calf$ is.

We denote by $\Aut(G,\calf)$ the subgroup of automorphisms of $G$ preserving $\calf$ and by $\Out(G,\calf)$ the subgroup of outer automorphisms of $G$ preserving $\calf$. An automorphism $\Phi \in \Aut(G,\calf)$ is \emph{fully irreducible} if no power of $\Phi$ fixes a proper free factor system of $(G,\calf)$.

A \emph{$(G,\calf)$-tree} is a tree equipped with an action of $G$ without inversion such that, for every $[A]\in \calf$, the group $A$ is elliptic in $T$. A \emph{Grushko $(G,\calf)$-tree} is a $(G,\calf)$-tree $T$ with trivial edge stabilisers and such that, for every $v \in VT$, the conjugacy class of the stabiliser $G_v$ of $v$ is trivial or contained in $\calf$.

Let $\calf'$ be a sporadic free factor system of $(G,\calf)$. There is a unique, up to unique $G$-equivariant homeomorphism, reduced Grushko $(G,\calf')$-tree $T_{\calf'}$, which we call the \emph{Bass-Serre tree} of $(G,\calf')$. (In this context, reduced means that no vertices have trivial stabiliser.) The tree $T_{\calf'}$ has a unique orbit of edges. The tree $T_{\calf'}$ is canonical in the sense that every element $\Phi \in \Aut(G,\calf')$ induces a $G$-equivariant homeomorphism of $T_{\calf'}$. Therefore, for every $\Phi \in \Aut(G,\calf')$, the group $G \rtimes_{\Phi} \ZZ$ acts by homeomorphisms on $T_{\calf'}$.

\subsection{Growth under an automorphism of a free product}

Let $(G,\calf)$ be a free product and let $T$ be a Grushko $(G,\calf)$-tree. We turn $T$ into a metric graph by assigning length $1$ to every edge of $T$. 

Let $g\in G$. The \emph{translation length} of $g$ in $T$ is $\lVert g\rVert_T =\inf_{x\in T} d(x,gx)$. The translation length of $g$ only depends on the conjugacy class of $g$.

Let $\Phi \in \Aut(G,\calf)$. An element $g\in G$ has \emph{$\lVert. \rVert_T$-polynomial growth under iteration of $\Phi$} if there exists $P \in \ZZ[X]$ such that, for every $n\in \NN$: $$\lVert \Phi^n(g)\rVert_T \leq P(n).$$ Note that any elliptic element of $G$ in $T$ has $\lVert.\rVert_T$-polynomial growth under iteration of $\Phi$. 

A subgroup $P$ of $G$ is a \emph{$\lVert. \rVert_T$-polynomial subgroup of $\Phi$} if there exists an automorphism $\Psi \in \Aut(G,\calf)$ contained in the outer class of some power of $\Phi$ such that $\Psi(P)=P$ and every element of $P$ has $\lVert. \rVert_T$-polynomial growth under iteration of $\Psi$. 

Let $\calp_T(\Phi)$ be the set of conjugacy classes of maximal $\lVert. \rVert_T$-polynomial subgroups of $\Phi$. By for instance \cite[Theorem~1.1]{Forester2002Rididity}, the set $\calp_T(\Phi)$ does not depend on $T$. 

When $\Phi$ is fully irreducible, the set $\calp_T(\Phi)$ satisfies some additional properties. Recall that a subgroup $A$ of $G$ is \emph{malnormal} if, for every $g\in G-A$, we have $A \cap gAg^{-1}=\{e\}$.

\begin{prop}\cite[Proposition~1.13]{dahmani2022relative}\label{prop:maximalpolysubgroups}
    Let $(G,\calf)$ be a nonsporadic free product and let $\Phi \in \Aut(G,\calf)$ be fully irreducible. Let $T$ be a Grushko $(G,\calf)$-tree.
    \begin{enumerate}
        \item The set $\calp_T(\Phi)$ is finite.
        \item For every $[A]\in \calp_T(\Phi)$, the subgroup $A$ is malnormal in $G$.
    \end{enumerate}
\end{prop}

Let $\calp=\{[P_1],\ldots,[P_{\ell}]\}$ be a finite set of conjugacy classes of malnormal subgroups of $G$. Let $\Phi \in \Aut(G)$ be an automorphism such that, for every $i \in \{1,\ldots,\ell\}$, there exists $g_i \in G$ such that $\mathrm{ad}_{g_i} \circ \Phi(P_i)=P_i$. The \emph{suspension of $\calp$} is the set $\{[P_i \rtimes_{\mathrm{ad}_{g_i} \circ \Phi} \ZZ]\}$ considered as a set of conjugacy classes of subgroups of $G\rtimes_\Phi\ZZ$. 

The following result is due to Dahmani--Li~\cite{dahmani2022relative}.

\begin{thm}\cite[Corollary~2.3]{dahmani2022relative}\label{thm:relativehyperbolicityfullyirred}
    Let $(G,\calf)$ be a nonsporadic free product and let $\Phi \in \Aut(G,\calf)$ be fully irreducible. Let $T$ be a Grushko $(G,\calf)$-tree and let  $\calp_T(\Phi)$ be the set of conjugacy classes of maximal $\lVert. \rVert_T$-polynomial subgroups of $\Phi$. There exists $n\in \NN$ such that the group $G \rtimes_{\Phi^n}\ZZ$ is hyperbolic relative to the suspension of $\calp_T(\Phi)$.
\end{thm}

In the rest of the section, we give a precise description of the set $\calp_T(\Phi)$ for a fully irreducible automorphism.

We first need a result, which can be found for instance in the work of Francaviglia--Martino--Syrigos~\cite{francaviglia2021action} concerning the existence of a limiting tree of a fully irreducible automorphism.

\begin{lemma}\cite[Lemma~2.14.1]{francaviglia2021action}\label{Lem:existenceRtree}
    Let $(G,\calf)$ be a nonsporadic free product and let $\Phi\in \Aut(G,\calf)$ be a fully irreducible automorphism. There exist a Grushko $(G,\calf)$-tree $S$, an $\RR$-tree $T$ equipped with an isometric action of $G$ and a constant $\lambda>1$ such that, for every $g \in G$, we have 
    \begin{equation}\label{Eq:limitingtree}
    \lim_{n\to \infty} \frac{1}{\lambda^n}\lVert \Phi^n g \rVert_S=\lVert g \rVert_T.    
    \end{equation}
    
\end{lemma}

Let $(G,\calf)$ be a free product. Note that, for any subgroup $A$ of $G$, the free factor system $\calf$ induces a free factor system $\calf|_A$ of $A$. Following the terminology of for instance Guirardel-Horbez~\cite[Definition~3.2]{Guirardelhorbez19}, an $\RR$-tree equipped with an isometric action of $G$ is an \emph{arational} $(G,\calf)$-tree if the following holds:

\begin{enumerate}
    \item the tree $T$ is not a Grushko $(G,\calf)$-tree;
    \item  for every $[A]\in \calf$, the group $A$ is elliptic in $T$;
    \item for every free factor system $\calf < \calf'$ and every $[A]\in \calf'$ such that $A$ is nonperipheral, the action of $A$ on its minimal tree in $T$ is a Grushko $(A,\calf|_A)$-tree.
\end{enumerate}

\begin{prop}\label{Prop:descriptionpolysubgroups}
    Let $(G,\calf)$ be a nonsporadic free product, let $\Phi\in \Aut(G,\calf)$ be a fully irreducible automorphism and let $S$ be a Grushko $(G,\calf)$-tree given by Lemma~\ref{Lem:existenceRtree}. For every $[P] \in \calp_S(\Phi)$, either $[P]\in \calf$ or $P$ is nonperipheral and infinite cyclic.
\end{prop}

\begin{proof}
    Let $T$ be the $\RR$-tree associated with $S$ given by Lemma~\ref{Lem:existenceRtree}. Note that, by Equation~\eqref{Eq:limitingtree}, for every $\lVert.\rVert_S$-polynomially growing element $g\in G$, we have $\lVert g \rVert_T=0$. 

    By~\cite[Theorems~3.4,4.1]{Guirardelhorbez19} the $\RR$-tree $T$ is an \emph{arational} $(G,\calf)$-tree. It has trivial arc stabilisers (because it is mixing~\cite[Lemma~4.9]{horbez2014tits}). Thus, for every $[P]\in \calp_S(\Phi)$, the group $P$ fixes a point in $T$.

    By~\cite[Lemma~4.6]{horbez2014tits}, using the fact that $T$ is arational, for every point $x \in T$, the stabiliser $G_x$ of $x$ is either peripheral or nonperipheral and infinite cyclic. Thus, for every $[P]\in \calp_S(\Phi)$, the elliptic subgroup $P$ is either peripheral or nonperipheral and infinite cyclic. By maximality of $P$, either $[P]\in \calf$ or $P$ is nonperipheral and infinite cyclic.
\end{proof}

\section{Actions on trees and JSJ decompositions}
\label{sec:JSJ}
\subsection{Tree of cylinders}
\label{subsec:trees_of_cylinders}

Let $G$ be a group acting on a tree $T$. In order to construct an acylindrical action of $G$ on a tree, we will modify the tree $T$ using the technology of \emph{tree of cylinders} introduced by Guirardel and Levitt~\cite{guirardel2011trees}.

Let $\cale$ be a class of subgroups of $G$, stable under conjugation. An \emph{$\cale$-tree} is a tree $T$ equipped with an action of $G$ without edge inversion and such that the stabiliser of any edge is contained in $\cale$. An equivalence relation $\sim$ on $\cale$ is \emph{admissible} if, for any $A,B \in \cale$, the following holds:

\begin{enumerate}
    \item for any $g \in G$, if $A \sim B$, then $gAg^{-1} \sim gBg^{-1}$;
    \item if $A \subseteq B$, then $A \sim B$;
    \item for every $\cale$-tree $T$, if $A \sim B$ and $A$ and $B$ are elliptic in $T$, then $\langle A,B\rangle$ is elliptic in $T$.
\end{enumerate}

The equivalence relation generated by inclusion is an admissible relation for every class of groups $\cale$ (see~\cite[Lemma~3.8]{guirardel2011trees}). If $\cale$ is the class of virtually infinite cyclic groups, then commensurability is an admissible equivalence relation, where two groups $A,B \in \cale$ are commensurable if $A \cap B$ has finite index in both $A$ and $B$.

Let $T$ be an $\cale$-tree and let $\sim$ be an admissible equivalence relation on $\cale$. If $e$ is an edge of $T$, we denote by $G_e$ its stabiliser in $T$. We define an equivalence relation $\sim_T$ on the set of edges of $T$ by setting, for all edges $e,e' \in ET$, $e \sim_T e'$ if and only if $G_e \sim G_{e'}$. A \emph{cylinder} $Y$ of $T$ is a $\sim_T$-equivalence class, seen as a subforest of $T$. A cylinder is in fact a subtree of $T$ (see~\cite[Lemma~4.2]{guirardel2011trees}).

\begin{defn}
Let $T$ be an $\cale$-tree. The \emph{tree of cylinders of $T$} is the bipartite tree $T_c$ whose vertex set $VT_c=V_0T_c \coprod V_1T_c$ is defined as follows:
\begin{enumerate}
    \item $V_0T_c$ is the set of vertices of $T$ belonging to at least two distinct cylinders;
    \item $V_1T_c$ is the set of cylinders of $T$;
    \item there is an edge between $v_0 \in V_0T$ and $v_1\in V_1T$ if the vertex in $T$ corresponding to $v_0$ belongs to the cylinder corresponding to $v_1$.
\end{enumerate}    
\end{defn}

The tree of cylinders of $T$ is a tree equipped with an action of $G$ without edge inversion.

\subsection{JSJ decompositions of one-ended relatively hyperbolic groups}

We now let $G$ be a hyperbolic group relative to a family $\calp=\{[P_1],\ldots,[P_n]\}$ of conjugacy classes of groups. Let $\calh$ be a set of conjugacy classes of subgroups of $G$ such that $G$ is one-ended relative to $\calp \cup \calh$ and let $\Phi \in \Aut(G,\calp \cup \calh)$. Let $G_\Phi$ be the suspension $G \rtimes_\Phi \ZZ$. In time, we will also assume the suspensions of the $P_i$ belong to $\FJCX$, and want to apply Theorem~\ref{Thm:FJCacylindrical} to $G_\Phi$ in order to prove that $G_\Phi \in \FJCX$. That is, we will construct a simplicial tree $T$ on which $G_\Phi$ acts acylindrically. The construction of the tree $T$ uses the theory of \emph{JSJ decomposition of groups}, which we now discuss, following the work of Guirardel-Levitt~\cite{guirardel2011trees,GuirardelLevitt2015, guirardellevitt2017jsj}.

A subgroup of $G$ is \emph{elementary} if it is finite, or virtually cyclic or conjugate into some $P_i$ with $i \in \{1,\ldots,n\}$. Let $\cala$ be the family of all \emph{infinite} elementary subgroups of $G$.  

Let $\sim_\cala$ be the equivalence relation on $\cala$ given by $A \sim_\cala B$ if $\langle A,B \rangle$ is elementary. The equivalence relation $\sim_\cala$ defines an admissible equivalence relation called \emph{coelementarity} (see~\cite[Lemma~3.4]{guirardel2011trees}). 

Let $\calh$ be any set of conjugacy classes of subgroups of $G$. Recall that an \emph{$(\cala,\calp \cup \calh)$-tree} is an $\cala$-tree $T$ such that, for every $[A]\in \calp \cup \calh$, the group $A$ is elliptic in $T$. We denote by $\Out(G,\calp \cup \calh^{(t)})$ the subgroup of $\Out(G,\calp \cup \calh)$ consisting of every $\psi \in \Out(G,\calp \cup \calh)$ such that, for every $[A] \in \calh$, there exists $\Psi \in \psi$ with $\Psi(A)=A$ and $\Psi|_A=\id_A$.

Let $T$ be a tree equipped with an action of $G$ by isometries with a finite number of orbits of edges. If $H$ is a subgroup of $\Out(G,\calp \cup \calh)$ preserving the $G$-equivariant homeomorphism class of a tree $T$, we denote by $H^0$ the finite index subgroup of $H$ acting trivially on $G \backslash T$. Note that, for every $v \in VT$, we have a homomorphism $H^0 \to \Out(G_v)$. 

Using \cite[Theorem~9.18]{guirardellevitt2017jsj}, \cite[Theorem~3.9]{GuirardelLevitt2015} and \cite[Proposition~6.1]{guirardel2011trees}, we have the following theorem.

\begin{thm}\cite{guirardel2011trees,GuirardelLevitt2015,guirardellevitt2017jsj}\label{Thm:existenceJSJ}
Let $G$ be a hyperbolic group relative to a family $\calp$ of non virtually cyclic groups. Let $\calh$ be any family of conjugacy classes of subgroups of $G$. Suppose that $G$ is one-ended relative to $\calp \cup \calh$. There exists a tree of cylinders $T_\calh$ for coelementarity equipped with an isometric action of $G$ such that:
\begin{enumerate}
\item the group $\Aut(G,\calp \cup \calh)$ preserves the $G$-equivariant homeomorphism class of $T_\calh$;
\item the action of $G$ on $T_\calh$ is $2$-acylindrical;
\item the number of orbits of edges is finite;
\item edge stabilisers are infinite elementary;
\item the tree $T_\calh$ is an $(\cala,\calp \cup \calh)$-tree;
\item vertex stabilisers corresponding to cylinders are elementary subgroups;
\item \label{thmJSJ:nonelementary_vertices} the vertex stabiliser $G_v$ of a noncylinder vertex $v$ satisfies one of the following:
\begin{itemize}
\item $G_v$ is nonelementary and \emph{Quadratically Hanging (QH) with finite fibre}; (see~\cite[Definition~5.13]{guirardellevitt2017jsj});
\item the vertex $v$ is nonelementary and \emph{rigid}: the stabiliser of $v$ is elliptic in every $(\cala,\calp \cup \calh)$-tree  (see~\cite[Definition~2.14]{guirardellevitt2017jsj}).
\end{itemize}
\item if $e_1,e_2$ are two distinct edges adjacent to the same nonelementary vertex, then $G_{e_1} \cap G_{e_2}$ is finite and $\langle G_{e_1},G_{e_2} \rangle$ is not elementary;
\item \label{thmJSJ:subgroups_in_H_are_rigid} if $[H] \in \calh$ is not elementary, then $H$ stabilises a unique rigid vertex. (Uniqueness follows from $(4)$. The fact that the vertex is rigid follows from~\cite[Definition~5.13(3)]{guirardellevitt2017jsj}, and the fact that QH vertices with finite fibre have virtually cyclic extended boundary subgroups).
\end{enumerate}

Moreover, if $\calh=\{[H_1],\ldots,[H_k]\}$ with every $H_i$ finitely generated: 

\begin{enumerate}[resume]
\item \label{thmJSJ:rigid_verts_finite_out} \cite[Theorem~3.9]{GuirardelLevitt2015} for every rigid vertex $v \in VT_\calh$, the homomorphism $\Out^0(G,\calp \cup \calh^{(t)}) \to \Out(G_v)$ is finite; 
\item \label{thmJSJ:edges_finite_out} for every edge $e \in ET_\calh$, the homomorphism $\Out^0(G,\calp \cup \calh^{(t)}) \to \Out(G_e)$ is finite.
\end{enumerate}
\end{thm}

When the family $\calh$ is explicit from context, we will refer to $T_\calh$ as $\Tcan$. (Our superscript convention here is certainly not standard: we use it because we will shortly need to discuss minimal invariant trees for subgroups coming from multiple underlying actions. This choice lets us write $\Tcan_H$, for instance, keeping both the tree and the subgroup conveniently in the notation.)

We now prove a sequence of lemmas in order to deduce acylindrical actions of $G_\Phi$ with $\Phi \in \Aut(G)$ on trees out of acylindrical actions of $G$. If $G_\Phi$ acts on a tree $T$, we denote by $F_\Phi$ the $G$-equivariant isometry of $T$ induced by $\Phi$.

Whilst we are certain the next lemma is known to experts we did not manage to locate it in the literature and so provide a proof.

\begin{lemma}
    \label{lem:elliptic_subgp_or_loxodromic_elet}
    If a group $G$ acts acylindrically on a tree $T$, then every subgroup $H\leqslant G$ also acts acylindrically on $T$.  Moreover, if $H$ has no loxodromic elements, then $H$ has a global fixed point on $T$.
\end{lemma}

\begin{proof}
    The first assertion is easy, since if for every $(K_1,K_2) \in \NN^2$ there is a path of length at least $K_1$ stabilised by a subgroup of $H$ of size at least $K_2$ then the same is true for $G$. For the second assertion, observe that a subgroup all of whose elements are elliptic but which does not have a global fixed point will be a proper ascending union $H_1 < H_2 < \cdots < H_n < \cdots$, where each $H_i$ stabilises an infinite ray. For each $K_2$ there will be a subgroup $H_i$ with greater cardinality, and so such an action cannot be acylindrical.
\end{proof}

\begin{lemma} 
    \label{lem:element_fixing_axis}
    Suppose $T$ is a $G$-tree, preserved by $\phi \in \Out(G)$, and for $\Phi \in \phi$ let $F_\Phi$ be the induced $G$-equivariant isometry of $T$. If $g$ is loxodromic and contained in $\Fix(\Phi)$, then there are $m \in \ZZ$ and $n \in \ZZ^*$ so that $g^mF_\Phi^n$ fixes pointwise the axis of $g$.
\end{lemma}

Note that it could be true that $\Phi^n$ or $F_{\Phi^n}$ is trivial, even if $n$ is not.

\begin{proof}
    Since $g$ is contained in $\Fix(\Phi)$, the isometry $F_\Phi$ commutes with $g$ and must preserve the axis of $g$ and its orientation. This implies that the axis is contained in the characteristic set of $F_\Phi$. 
    Thus, we have a homomorphism $\Lambda \colon \ZZ^2 \to \langle g,F_{\Phi} \rangle \to \RR$ given by the translation length on the axis of $g$. Since $T$ is a simplicial tree, the image of $\Lambda$ is a discrete subset of $\RR$. Thus, the image of $\Lambda$ is cyclic and the kernel of $\Lambda$ is nontrivial. Since it cannot be contained in $\langle g \rangle$, there are $m \in \ZZ$ and $n \in \ZZ^*$ such that $g^mF_{\Phi^n}$ fixes pointwise the axis of $g$.
\end{proof}

\begin{lemma}\label{lem:PromoteAcylindricalactionsv2}
    Let $K\geq 1$, and let $\phi=[\Phi] \in \Out(G)$. Suppose that $G_\Phi$ acts on a tree $T$ with finitely many orbits of edges and that the action of $G$ on $T$ is $K$-acylindrical. Suppose that there exists $J \in \NN^*$ so that for every geodesic path $\gamma$ of length $J$ and every automorphism $\Psi \in \phi$ such that $F_\Psi$ preserves $\gamma$, there exist a vertex $v$ of $\gamma$ and $g \in G_v$ of infinite order fixed by a power of $\Psi$.

    \begin{enumerate}
        \item Let $n \in \NN^*$ and let $\Psi \in \phi^n$. Suppose that $F_\Psi$ fixes pointwise a geodesic edge path of length at least equal to $K + 2J +1$. There exists $N \in \NN^*$ such that $\Psi^N$ fixes elementwise a nonabelian free group $L\subseteq G$ consisting of loxodromic elements of $T$. Moreover, for every $g \in L$, the isometry $F_{\Psi^N}$ fixes elementwise the axis of $g$.
        \item The action of $G_\Phi$ on $T$ is acylindrical if and only if for every $n \in \NN$ and every $\Psi$ in the outer class of $\Phi^n$, the group $\Fix(\Psi)$ is elliptic in $T$.
\end{enumerate}
\end{lemma}

\begin{proof}
\noindent{$(1)$ } Suppose that $F_\Psi$ fixes pointwise a geodesic edge path $\gamma$ of length $K+2J+1$. Thus, the path $\gamma$ is not reduced to an edge and the isometry $F_\Psi$ fixes the initial and the terminal paths $\gamma_1,\gamma_2$ of $\gamma$ length $J$. Since the action of $G$ on $T$ is $K$-acylindrical, for all vertices $v_1$ of $\gamma_1$ and $v_2$ of $\gamma_2$, the intersection $G_{v_1}\cap G_{v_2}$ is finite. 

Let $i \in \{1,2\}$. Note that $F_{\Psi}$ preserves $\gamma_i$. By hypothesis, there exist $N_i \in \NN^*$, a vertex $v_i \in \gamma_i$ and an infinite order element $g_i \in G_{v_i}$ which is fixed by $\Psi^{N_i}$.

Let $N=N_1N_2$ and let $L=\langle g_1,g_2 \rangle \subseteq \Fix(\Psi^N)$. Since $G_{v_1}\cap G_{v_2}$ is finite, we have $\Fix(g_1) \cap \Fix(g_2)=\varnothing$.  By standard ping-pong arguments, the group $L$ is a non-abelian free group which contains a (non-abelian) subgroup consisting of loxodromic elements of $T$. This proves the first part of Assertion~$(1)$.

Let $g \in L$ be loxodromic. Since $\Psi^N(g)=g$, the isometry $F_{\Psi^N}$ commutes with $g$. In particular, the axis of $g$ is contained in the characteristic set of $F_{\Psi^N}$. Since $F_{\Psi^N}$ is elliptic, this implies that $F_{\Psi^N}$ fixes pointwise the axis of $g$. This concludes the proof of Assertion~$(1)$.

\medskip

\noindent{$(2)$ } By Lemma~\ref{lem:elliptic_subgp_or_loxodromic_elet}, $\Fix(\Psi)$ is either elliptic or contains an element acting loxodromically. Suppose that there exist $n \in \NN^*$, $\Psi \in \phi^n$ and $g \in \Fix(\Psi)$ such that the action of $g$ on $T$ is loxodromic. By Lemma~\ref{lem:element_fixing_axis}, there exist $m \in \ZZ$ and $n \in \ZZ^*$ such that $g^mF_{\Psi^n}$ fixes pointwise the axis of $g$. Hence the action of $G_\Phi$ is not acylindrical.

Conversely, suppose that the action of $G_\Phi=\langle G,t \rangle$ on $T$ is not acylindrical. Since the action of $G$ on $T$ is $K$-acylindrical, there exist $g \in G$ and $k \in \NN^*$ such that the element $gt^k$ fixes an edge path of length $K+2J+1$. Let $\Psi \in \phi^k$ be the automorphism corresponding to $gt^k$. By Assertion~$(1)$, some power of $\Psi$ fixes a loxodromic element of $G$. This proves Assertion~$(2)$ and concludes the proof.
\end{proof}


\section{The periodic JSJ decomposition}\label{sec:periodicJSJ}
We now specialise to the tree we will use to prove that suspensions of one-ended relatively hyperbolic groups (under reasonable assumptions on the parabolic subgroups) satisfy the Farrell--Jones conjecture. Let $G$ be a hyperbolic group relative to $\calp$. Let $\calh_0$ be a finite set of conjugacy classes of finitely generated subgroups of $G$ such that $G$ is one-ended relative to $\mathcal{P} \cup \calh_0$. Let $\Phi \in \Aut(G,\mathcal{P} \cup \calh_0^{(t)})$. We explain in the following section the construction of $G$-trees which are naturally associated with $\Phi$. Note that for the results needed for the Farrell--Jones conjecture, we take $\calh_0$ to be empty; but we will need it to be non-empty for an argument in \cref{subsec:slender_peripherals}.

\subsection{Trees associated with an automorphism of a one-ended relatively hyperbolic group}\label{section:constructiontrees}

Recall that, if $\Phi \in \Aut(G)$, we denote by $\mathrm{Per}(\Phi)$ the subgroup of $G$ consisting of all $g \in G$ such that there exists $n \in \NN^*$ with $\Phi^n(g)=g$. Let $\phi=[\Phi] \in \Out(G)$. We denote by $\mathrm{NP}(\phi)$ the set of all representatives $\Phi \in \phi$ such that $\mathrm{Per}(\Phi)$ is not an elementary subgroup. If $\phi \in \Out(G)$, we set $\mathrm{Per}_{\mathrm{NP}}(\phi)=\{[\mathrm{Per}(\Phi)]\}_{\Phi \in \mathrm{NP}(\phi)}$. 

We work with three $\phi$-invariant trees for $G$. The first is the canonical JSJ tree $\Tcan$ for the set $\calp \cup \calh_0$, the second is the tree $\Tper$ obtained by applying \cref{Thm:existenceJSJ} with $\calh = \calh_0 \cup \bigcup_{n \in \NN^*}\mathrm{Per}_{\mathrm{NP}}(\phi^n)$. Note that if we restrict to $\Phi \in \Aut(G, \calp \cup \calh_0^{(t)})$ then each $H \in \calh_0$ is contained in some periodic subgroup for $\phi$; this explains its limited effect on the proofs below. The third one is obtained from $\Tcan$ by blowing-up JSJ trees at QH with fibre vertices. The following lemmas motivate the construction.

\begin{lemma}\label{Lem:equalityautofixnonelementary}
    Let $G$ be a hyperbolic group relative to $\calp$ and let $\phi=[\Phi] \in \Out(G,\calp)$.  Let $n \in \NN^*$ and let $\Psi,\Theta \in \phi^n$ be such that $\Psi$ and $\Theta$ fix elementwise the same nonelementary subgroup $H$ of $G$. 

    There exists $N \in \NN^*$ such that, $\Psi^{N}=\Theta^{N}$.
\end{lemma}

\begin{proof}
    Since $\Psi$ and $\Theta$ fix $H$ elementwise, $\Psi$ and $\Theta$ differ by an inner automorphism in the centraliser of $H$. Since $H$ is nonelementary, its centraliser is finite (see for instance~\cite[Theorem~4.19]{osin2006relatively}). Thus, up to taking powers of $\Psi$ and $\Theta$ fixing elementwise the centraliser of $H$, we have $\Psi=\mathrm{ad}_g \circ \Theta$ where $g \in C_G(H)$ and $g \in \Fix(\Theta)$. Thus, for every $m \geq 1$, we have $\Psi^m=\mathrm{ad}_{g^m}\circ \Theta^m$. As $g$ is finite order,  there exists $N \in \NN^*$ such that $\Psi^{N}=\Theta^{N}$.
\end{proof}

\begin{lemma}
    \label{lem:orbifold_periodic_is_fg}
    Suppose $G$ is the vertex group of a QH with fibre vertex of $\Tcan$ and let $\phi=[\Phi] \in \Out(G)$. 
    
    \begin{enumerate}
        \item The group $\Per(\Phi)$ is finitely generated, and there is some $k \in \NN$ so that $\Per(\Phi)=\Fix(\Phi^k)$.
        \item There exists $k \in \NN^*$ so that if $[g]$ is a periodic conjugacy class of $\phi$ then $[g]$ is fixed by $\phi^k$.
        \item As $\Phi$ varies over the outer classes $\phi^n$ with $n\in \NN^*$, there are only finitely many conjugacy classes of maximal (for inclusion) periodic subgroups $\Per(\Phi)$.
        \item There exists $k \in \NN^*$ so that if $[K]$ is a conjugacy class of periodic subgroups of $\phi$, then $[K]$ is fixed by $\phi^k$.
    \end{enumerate}
\end{lemma}

This result does not seem surprising, and in fact the same statement is true for all hyperbolic groups (see \cref{Coro:Slenderperiodicisfixedoneended} and \cite{guirardel2016mccool} for the torsion free case). However, this special case is necessary to begin the arguments on JSJ decompositions we use throughout this section, including to prove the general statement.

\begin{proof}
    We first prove Lemma~\ref{lem:orbifold_periodic_is_fg} when $G$ is a hyperbolic 2-orbifold. Recall that hyperbolic 2-orbifolds are good, and let $H$ be a characteristic finite index subgroup of $G$ corresponding to an orientable surface cover of the orbifold. (This can be obtained by taking the characteristic core of the subgroup corresponding to any such cover, since $G$ is finitely generated.) In particular, $\Phi$ preserves $H$. We now consider two periodic subgroups: $\Per_G(\Phi) \leq G$ and its subgroup $\Per_H(\Phi|_H) \leq H$. If $g_1$ and $g_2$ are elements of $\Per_G(\Phi)$ representing the same coset of $H$ in $G$, then in fact they represent the same coset of $\Per_H(\Phi|_H)$ in $\Per_G(\Phi)$, so this is a finite index subgroup. The restriction $\Phi|_H$ can be represented by an element of the mapping class group of the surface, and it follows from \cite{ivanov1992subgroups} that periodic subgroups here are finitely generated. Finite generation is a commensurability invariant, so the same is true of $\Per_G(\Phi)$. 
    Then taking a sufficiently high power to fix every element of a finite generating set shows that $\Per(\Phi)=\Fix(\Phi^k)$.

    Let $[g]$ be a $\phi$-periodic conjugacy class. If $g$ has finite order, then, as there exists finitely many conjugacy classes of finite order elements in $G$, some power of $\phi$ fixes $[g]$. Suppose now that $g$ has infinite order, and let $t \in \NN^*$ be such that $g^{t} \in H$. Since $G$ is hyperbolic, $g^{t}$ has finitely many $t$-th roots in $G$, the number of such roots depending only on the finite numbers of orders of the finite subgroups of $G$. Thus, if $\ell \in \NN^*$ is such that $\phi^{\ell} \in \Out(H)$ fixes the conjugacy class of $g^t$, then a power of $\phi^\ell$ fixes the conjugacy class of $g$ and this power does not depend on $g$. If $H$ is a free group, then the existence (and uniformity) of $\ell$ follows from the work of Handel--Mosher~\cite[Theorem~II.4.1]{handel2020subgroup}. If $H$ is the fundamental group of a closed orientable surface, this follows from the work of Ivanov \cite{ivanov1992subgroups}.
        

    Since the third statement is true for free and surface groups (by Ivanov~\cite{ivanov1992subgroups} for the surface case, Bestvina--Handel~\cite{bestvina1992train} for the free case with noncyclic periodic subgroups, and for instance Feighn--Handel~\cite[Corollary 11.1]{FeighnHandel2018} 
    for the general case), 
    it will suffice to bound the number of subgroups $\Per_G(\Phi)$ containing (with finite index) a given restriction $\Per_H(\Phi|_H)$. If this is non-elementary, then it follows from \cref{Lem:equalityautofixnonelementary} 
    that any two automorphisms in $\phi$ fixing it have a common power, and hence the same periodic subgroups. (First replace $\Phi^k\ad(g_1)$ and $\Phi^\ell\ad(g_2)$ with their $\ell$-th and $k$-th powers respectively, so they represent the same outer automorphism, then another power so as to fix the common non-elementary subgroup $\Per_H(\Phi|_H)$, then apply the lemma as written.)

    Now assume $\Per_H(\Phi|_H)$ is elementary, and we want to control the periodic subgroups of $\Phi$ in $G$ restricting to it. Recall that in a hyperbolic group, every virtually cyclic subgroup is contained in a unique maximal one, and let $M$ be the maximal virtually cyclic subgroup containing $\Per_H(\Phi|_H)$. Since $\Phi$ preserves $\Per_H(\Phi|_H)$, it must also preserve $M$, and we consider the induced automorphism of $M$. By for instance \cite[Lemma 6.6]{MinasyanOsin2010OutVCyc}, $\Out(M)$ is finite, and so passing to a power $\Phi^k$ the induced automorphism is inner. Composing with an inner automorphism coming from $M$, some representative $\Psi$ of $\Phi^k$ fixes $M$; in particular $M$ is itself a periodic subgroup.

    Note that this inclusion between the periodic subgroup for $\Phi$ and the one for $\Psi$ can stay proper at all powers. For instance take $\Phi$ to be the infinite order inner automorphism of $D_\infty$ and $\Psi$ to be the identity automorphism of $D_\infty$.

    Finally, as there exist only finitely many conjugacy classes of finite subgroups in $G$, there exist only finitely many conjugacy classes of finite periodic subgroups for any power of $\phi$.    

    The final property follows from taking a high enough power to fix (up to composing with appropriate inner automorphisms) the generating sets of a representative of each conjugacy class.

    Let $G_v$ be the vertex group of a QH with finite fibre vertex. Let $F$ be a finite normal subgroup of $G_v$ such that $G_v/F$ is isomorphic to the fundamental group $\pi_1(\Sigma_v)$ of a 2-orbifold $\Sigma_v$. The group $F$ is the unique maximal finite normal subgroup of the hyperbolic group $G_v$ and is in particular characteristic. Thus, every $\phi \in \Out(G_v)$ induces an element $\phi|_{\Sigma_v} \in \Out(\pi_1(\Sigma_v))$. 

    Let $\Phi \in \Aut(G_v)$. Let $\ell \in \NN^*$ be the integer associated with $\Phi|_{\pi_1(\Sigma_v)}$ which satisfies both Assertions~$(1)$ and $(2)$. Let $g \in \Per(\Phi)$. Then the image of $g$ in $\pi_1(\Sigma_v)$ is fixed by $\Phi^\ell|_{\pi_1(\Sigma_v)}$. Thus, $\Phi^\ell$ preserves the left coset $gF$. As $F$ is finite, the automorphism $\Phi^{\ell|\Aut(F)|}$ acts trivially on $F$. Thus, $\Phi^{\ell|\Aut(F)|}(g)=gf$ for some $f \in F$ fixed by $\Phi^{\ell|\Aut(F)|}$. As $f^{|F|}=1$, we see that $\Phi^{\ell|\Aut(F)||F|}$ fixes $g$. This proves Assertion~$(1)$. Similarly, suppose that $[g]$ is a periodic conjugacy class. Then $\Phi^\ell|_{\pi_1(\Sigma_v)}$ preserves the conjugacy class in $\pi_1(\Sigma_v)$ induced by $g$. Thus, $\Phi^\ell$ sends $g$ to $hk_1gk_2h^{-1}$ with $h\in G_v$ and $k_1,k_2 \in F$. Let $\Psi^\ell=\mathrm{ad}_{h^{-1}} \circ \Phi^\ell$. Then $\Psi^\ell$ sends $g$ to $k_1gk_2$. Note that $\Psi^{\ell|\Aut(F)|}$ acts trivially on $F$ and sends $g$ to $k_1'gk_2'$ with $k_1',k_2' \in F$. Thus, $\Psi^{\ell|\Aut(F)||F|}$ fixes $g$ and $\Phi^{\ell|\Aut(F)||F|}$ fixes the conjugacy class of $g$. This proves Assertion~$(2)$. 

    For the third assertion, notice that the periodic subgroups of the induced action on $\pi_1(\Sigma_v)$ contain (by passing to a finite index surface subgroup, and if necessary then an infinite index one) a preserved (and periodic) free group. This splits back to $G_v$, and since the arguments given earlier used only \cref{Lem:equalityautofixnonelementary} and properties of virtually cyclic subgroups of hyperbolic groups, they apply equally well here (note that it is enough to take \emph{any} non-elementary periodic subgroup to apply \cref{Lem:equalityautofixnonelementary}). Again, the final assertion follows by taking a high enough power to fix (up to composing with appropriate inner automorphisms) the generating sets of a representative of each conjugacy class. \qedhere


    
\end{proof}

For each non-elementary vertex $v$ of $\Tcan$, let $t(v)$ be the \emph{first return power}: the smallest positive power so that $\phi$ preserves the conjugacy class of $G_v$. Note that $G_v$ is self-normalising, and so $\phi^{t(v)}$ induces a well defined outer automorphism of $G_v$, which we denote by $\phi^{t(v)}|_{G_v}$. Then let \[\calh_{QH}=\bigcup_{\substack{v \in V\Tcan/G \\ v \text{ is QH}}} \quad\bigcup_{n \in \NN^*}\mathrm{Per}_{\mathrm{NP}}(\phi^{n t(v)}|_{G_v}).\] This set is finite by Lemma~\ref{lem:orbifold_periodic_is_fg}~$(3)$. Let $T^\phi$ be the JSJ tree of $G$ over elementary subgroups relative to $ \calp \cup \calh_0 \cup \calh_{QH}$ given by Theorem~\ref{Thm:existenceJSJ}. Note that, for every QH with fibre vertex $v \in V\Tcan$ and every $H_v \in \bigcup_{n \in \NN^*}\mathrm{Per}_{\mathrm{NP}}([\phi^n|_{G_v}])$, there exists $H \in \bigcup_{n \in \NN^*}\mathrm{Per}_{\mathrm{NP}}([\phi^n])$ such that $H_v=H \cap G_v$. The tree $T^\phi$ is $\phi$-invariant and, moreover, the tree $T^\phi$ is \emph{compatible} with $\Tper$.

We briefly recall some facts about compatibility. Two $(\cala,\calp)$-trees $T$ and $T'$ are \emph{compatible} if there exists an $(\cala,\calp)$-tree $U$ such that both $T$ and $T'$ are obtained from $U$ by collapsing some orbits of edges. By~\cite[Proposition~A.26]{guirardellevitt2017jsj}, there exists a unique such minimal tree $U$ which refines $T$ and $T'$. The tree $U$ satisfies the following properties: a subgroup $H$ of $G$ stabilises a point in $U$ if and only if $H$ stabilises a point in both $T$ and $T'$. Moreover, for every edge $e \in EU$, the image of $e$ in either $T$ or $T'$ is not reduced to a point. 

By universality of $\Tcan$ (see~\cite[Corollary~9.18(3)]{guirardellevitt2017jsj}), the trees $\Tcan$ and $\Tper$ are compatible. Since the set $\calh_{QH}$ consists of conjugacy classes of elliptic subgroups in both $\Tcan$ and $\Tper$, by universality of $T^\phi$, the trees $\Tcan$, $T^\phi$ and $\Tper$ are pairwise compatible. 

Note that, unlike $\Tcan$, the tree $T^\phi$ is not necessarily preserved by every element of $\Out(G,\calp \cup \calh_0)$. However, it is preserved by $\phi$, which is sufficient for our considerations. Moreover, the tree $T^\phi$ is not necessarily compatible with every $(\cala,\calp)$-tree, but we will only need the fact that it is compatible with $\Tper$.  The fact that we replace $\Tcan$ by $T^\phi$ is due to the following lemmas.

\begin{lemma}\label{lem:Tphiacylindrical}
    Let $v \in VT^\phi$ be either rigid or QH with fibre and let $e_1,e_2 \in ET^\phi$ be two distinct edges adjacent to $v$. Then $G_{e_1}\cap G_{e_2}$ is finite and $\langle G_{e_1},G_{e_2}\rangle$ is not elementary. 
    In particular, the tree $T^\phi$ is $2$-acylindrical. 
\end{lemma}

\begin{proof}
This is Theorem~\ref{Thm:existenceJSJ}~$(2),(8)$.
\end{proof}

Recall that we denote by $\langle \phi \rangle^0$ the finite index subgroup of $\langle \phi \rangle$ which acts trivially on the quotient graph $G\backslash T^\phi$. For every $v\in VT^\phi$, we then have a well-defined homomorphism $\langle \phi \rangle^0 \to \Out(G_v)$.

\begin{lemma}\label{lem:Tphirigidperiodic}
    Let $v \in VT^\phi$ be rigid. The map $\langle \phi \rangle^0 \to \Out(G_v)$ has finite image.
\end{lemma}

\begin{proof}
By Lemma~\ref{lem:orbifold_periodic_is_fg}~$(1)$, $\calh_{QH}$ is a finite set of finitely generated groups, each periodic for some representative of a power of $\phi$, and so there exists $k>1$ such that $\phi^k \in \Out(G,\calp \cup \calh_0^{(t)} \cup \calh_{QH}^{(t)})$. The result now follows from Theorem~\ref{Thm:existenceJSJ}~$(10)$.
\end{proof}

Recall that $\Tcan$ and $T^\phi$ are compatible. Let $\Tref_0$ be their minimal refinement. Since both $\Tcan$ and $T^\phi$ are $(\cala,\calp \cup \calh_0 \cup \calh_{QH})$-trees, so is $\Tref_0$.

\begin{lemma}\label{lem:Tref0_is_JSJ}
    The tree $\Tref_0$ is a JSJ tree over elementary subgroups relative to $\calp \cup \calh_0 \cup \calh_{QH}$. In particular, $\Tref_0$ and $T^\phi$ have the same set of elliptic subgroups.
\end{lemma}

\begin{proof}
    We check~\cite[Definition~2.11]{guirardellevitt2017jsj}. We first prove that $\Tref_0$ is \emph{universally elliptic in $(\cala,\calp \cup \calh_0 \cup \calh_{QH})$-trees}, that is every edge group of $\Tref_0$ is elliptic in every $(\cala,\calp \cup \calh_0 \cup \calh_{QH})$-trees. Let $G_e$ be an edge group of $\Tref_0$. By minimality of $\Tref_0$, the group $G_e$ is an edge group of either $\Tcan$ or $T^\phi$. As $\Tcan$ is a JSJ $(\cala,\calp \cup \calh_0)$-tree, it is universally elliptic in $(\cala,\calp \cup \calh_0)$-trees, hence it is universally elliptic in $(\cala,\calp \cup \calh_0 \cup \calh_{QH})$-trees. Similarly, $T^\phi$ is universally elliptic in $(\cala,\calp \cup \calh_0 \cup \calh_{QH})$-trees. Thus, $G_e$ is  elliptic in every $(\cala,\calp \cup \calh_0 \cup \calh_{QH})$-trees and $\Tref_0$ is  universally elliptic in $(\cala,\calp \cup \calh_0 \cup \calh_{QH})$-trees.

    We now prove that $\Tref_0$ \emph{dominates every universally elliptic $(\cala,\calp \cup \calh_0 \cup \calh_{QH})$-trees}, that is, every vertex group of $\Tref_0$ is elliptic in every universally elliptic $(\cala,\calp \cup \calh_0 \cup \calh_{QH})$-tree. Since $T^\phi$ is a JSJ $(\cala,\calp \cup \calh_0 \cup \calh_{QH})$-tree, it dominates every universally elliptic $(\cala,\calp \cup \calh_0 \cup \calh_{QH})$-trees. Since $\Tref_0$ collapses onto $T^\phi$, every vertex group of $\Tref_0$ is elliptic in $T^\phi$ hence is elliptic in every universally elliptic $(\cala,\calp \cup \calh_0 \cup \calh_{QH})$-tree. This shows that $\Tref_0$ dominates every universally elliptic $(\cala,\calp \cup \calh_0 \cup \calh_{QH})$-trees. Together with the above paragraph, we obtain that $\Tref_0$ is a JSJ tree over elementary subgroups relative to $\calp \cup \calh_0 \cup \calh_{QH}$.
\end{proof}

\begin{lemma}\label{lem:additionalpropTphi}
    Let $v \in T^\phi$ and let $[H]\in \per$.
\begin{enumerate}
    \item Suppose that $v$ is QH with fibre. Then $H \cap G_v$ is elementary.
    \item Suppose that $v$ is rigid. If $H \cap G_v$ is nonelementary, then $G_v \subseteq H$.
\end{enumerate}
\end{lemma}

\begin{proof}
    Suppose first that $v$ is QH with fibre. Suppose towards a contradiction that $H \cap G_v$ is nonelementary. By Lemma~\ref{lem:Tref0_is_JSJ}, the group $G_v$ is an elliptic group of the refined tree $\Tref_0$. We still denote by $v$ the vertex in $\Tref_0$ fixed by $G_v$. Let $p \colon \Tref_0 \to \Tcan$ be the natural projection. Since $G_v$ is a flexible vertex, the image $p(v)$ is not a rigid vertex of $\Tcan$ (recall that every rigid vertex stabiliser of $\Tcan$ is also elliptic in every $(\cala,\calp \cup \calh_0 \cup \calh_{QH})$-tree). Hence $p(v)$ is a flexible vertex of $\Tcan$. Therefore, if $H \cap G_v$ is nonelementary, then $H \cap G_v \in \calh_{QH}$. This contradicts Theorem~\ref{Thm:existenceJSJ}~$(9)$. Hence $H \cap G_v$ is elementary.

    Suppose now that $v$ is rigid. Let $\Psi \in \phi$ be such that $\mathrm{Per}(\Psi)=H$. Suppose that the intersection $H \cap G_v$ is nonelementary. Then $\Psi$ preserves $G_v$ since $v$ is the unique vertex in $T^\phi$ fixed by $H\cap G_v$. By~\Cref{lem:Tphirigidperiodic}, there exists $k \in \NN^*$  such that $\Psi^k$ acts as a global conjugation on $G_v$ by an element $g\in G_v$. Taking a larger $k$ if necessary, we may also assume that $\Psi^k$ acts trivially on a nonelementary subgroup $H'\subseteq H \cap G_v$. 
    
    Note that $\Theta=\mathrm{ad}_{g^{-1}}\circ \Psi^k$ acts trivially on $G_v$. Thus, $\Psi^k$ and $\Theta$ acts trivially on the same nonelementary subgroup $H'\subseteq G_v$. By~\Cref{Lem:equalityautofixnonelementary}, there exists $N \in \NN$ such that $\Psi^{kN}=\Theta^N$. 
   In particular, $\Psi$ has a power which fixes $G_v$ elementwise. Thus, we have $G_v \subseteq H$.
\end{proof}

\begin{remark}\label{Remark:Phielliptic}
Note that, for every $\Phi \in \mathrm{NP}(\phi)$, the isometry $F_\Phi$ of $T^\phi$ is elliptic. Indeed, if $F_{\Phi}$ is loxodromic, then $\Phi$ can only fix an element $g \in G$ which is loxodromic and whose axis is the same as the one of $F_{\Phi}$. Since the action of $G$ on $T^\phi$ is acylindrical, the element $g$ is contained in a unique maximal virtually cyclic subgroup. In particular, $\mathrm{Per}(\Phi)$ is a virtually cyclic group and $\Phi \notin \mathrm{NP}(\phi)$.
\end{remark}

The rest of the section is dedicated to the proof of some properties of the set $\per$ and of the action of $G_\Phi$ on $T^\phi$. To this end we prove that $\mathrm{Per}_{\mathrm{NP}}(\phi)$ is finite (see Lemma~\ref{Lemma:Perfinite}). We need the following lemmas regarding the intersection of characteristic sets of isometries in $T^\phi$.

\begin{lemma}\label{Lemautomorphismfixingadjacentedges}
Let $G$ be a hyperbolic group relative to $\calp$. Let $\calh_0$ be a finite set of conjugacy classes of finitely generated subgroups of $G$ such that G is one-ended relative to $\calp \cup \calh_0$. 

Let $v \in VT^\phi$ be non-elementary and let $e_1$ and $e_2$ be distinct edges adjacent to $v$. Let $\phi \in \Out(G,\calp)$. Suppose that there is a representative $\Phi \in \phi$ such that, for every $i=1,2$, we have $\Phi(G_{e_i})=G_{e_i}$. Then $v$ is rigid
and $G_v$ is fixed elementwise by some power of $\Phi$.
\end{lemma}

\begin{proof}

We first prove that $v$ is rigid. Indeed, suppose towards a contradiction that $v$ is QH with finite fibre. Then $G_{e_1}$ and $G_{e_2}$ are virtually cyclic. Thus, for every $i \in \{1,2\}$, the automorphism $\Phi$ has a power $\Phi^k$ fixing an infinite order element $g_i \in G_{e_i}$. By~\Cref{lem:Tphiacylindrical} and as the groups 
$G_{e_1}$ and $G_{e_2}$ are virtually cyclic, the group $\langle g_1,g_2 \rangle$ is a non-elementary subgroup. As $\langle g_1,g_2 \rangle \subseteq \mathrm{Per}(\Phi) \cap G_v$, this contradicts \Cref{lem:additionalpropTphi}. Thus, the vertex $v$ is rigid.


By~\Cref{lem:Tphirigidperiodic}, after taking a power $\Phi^\ell$, it acts on $G_v$ as global conjugation by an element $g \in G_v$. 

We claim that, after taking a further power of $\Phi$, the element $g$ is trivial. Indeed, note that, by \Cref{lem:Tphiacylindrical} the stabiliser of an edge adjacent to $v$ is almost malnormal in $G_v$: for every edge $e'$ of $T^\phi$ adjacent to $v$ and every $g' \in G_v$, if $g'G_eg'^{-1} \cap G_e$ is infinite then $g' \in G_e$. Moreover, if $e$ and $e'$ are two distinct edges adjacent to $v$, then $G_{e} \cap G_{e'}$ is finite. Since $\Phi$ preserves $G_{e_1}$ and $G_{e_2}$, the power $\Phi^\ell$ must act by conjugating by an element of the finite intersection $G_{e_1} \cap G_{e_2}$. This becomes trivial on taking a further power of $\Phi$, which concludes the proof of the claim.

Thus we have shown that $\Phi$ has a power fixing $G_v$ elementwise, which concludes the proof. \qedhere

\end{proof}

\begin{corollary}\label{Coro:stabedgeinperiodicsubgroup}
     Let $G$ be a hyperbolic group relative to $\calp$. Let $\calh_0$ be a finite set of conjugacy classes of finitely generated subgroups of $G$ such that G is one-ended relative to $\calp \cup \calh_0$. Let $\phi=[\Phi] \in \Out(G,\calp \cup \calh_0^{(t)})$. Let $[H]=[\mathrm{Per}(\Phi)]\in \per$ and let $T_H^\phi$ be the minimal $H$-invariant subtree of $T^\phi$. 
     
     If $T_H^\phi$ contains an edge $e$, then the endpoints of $e$ are respectively elementary and rigid. If $v$ is the rigid endpoint of $e$, then $G_e \subseteq G_v\subseteq H$ and $G_v$ is fixed elementwise by some power of $\Phi$.
\end{corollary}

\begin{proof}
Note that, since $T^\phi$ is bipartite, the endpoints of $e$ are respectively elementary and rigid or QH with fibre. It suffices to prove that an endpoint $v$ of $e$ cannot be QH with fibre.

Since $T_{H}^\phi$ contains an edge, it follows that $H$ is not elliptic in $T^\phi$. Thus, $T^\phi_H$ is the union of the axes of elements of $H$. Then, $e$ is contained in the axis of an element $g$ of $H$. Recalling the bipartite structure of $T^\phi$, let $v$ be the rigid or QH vertex adjacent to $e$. Then there exists an edge $e' \neq e$ adjacent to $v$ and contained in the axis of $g$. Let $\Phi \in \phi$ be such that $\mathrm{Per}(\Phi)=H$ and let $N \in \NN^*$ be such that $\Phi^N(g)=g$. Since $\Phi^N$ is elliptic in $T^\phi$ by \Cref{Remark:Phielliptic}, it fixes pointwise the axis of $g$. In particular, it fixes $e$ and $e'$. By Lemma~\ref{Lemautomorphismfixingadjacentedges}, we see that $v$ is rigid and that $G_v$ is fixed elementwise by a power of $\Phi$. 
\end{proof}

\begin{corollary}\label{coro:noQHwithfibrevertex}
     Let $G$ be a hyperbolic group relative to $\calp$. Let $\calh_0$ be a finite set of conjugacy classes of finitely generated subgroups of $G$ such that G is one-ended relative to $\calp \cup \calh_0$. Let $\phi=[\Phi] \in \Out(G,\calp \cup \calh_0^{(t)})$. Let $[H]\in \per$ and let $T_H^\phi$ be the minimal $H$-invariant subtree of $T^\phi$. 

     The tree $T_H^\phi$ does not contain a QH with fibre vertex. 
\end{corollary}

\begin{proof}
    Suppose first that $T_H^\phi$ is reduced to a point $v$. Then $H \subseteq G_v$ and $v$ is not QH with fibre by \Cref{lem:additionalpropTphi}. Suppose now that $T_H^\phi$ is not reduced to a point. Then any vertex $v$ of $T_H^\phi$ is adjacent to an edge and the result follows from \Cref{Coro:stabedgeinperiodicsubgroup}.
\end{proof}


\begin{lemma}\label{Lem:equalityautomorphismfixingaxis}
Let $G$ be a hyperbolic group relative to $\calp$. Let $\calh_0$ be a finite set of conjugacy classes of finitely generated subgroups of $G$ such that G is one-ended relative to $\calp \cup \calh_0$. Let $\phi=[\Phi] \in \Out(G,\calp \cup \calh_0^{(t)})$. 

Let $n \in \NN^*$ and let $\Psi,\Theta \in \phi^n$ be such that $F_\Psi$ and $F_\Theta$ are elliptic isometries of $T^\phi$. Suppose that there exist $g,h \in G$ loxodromic in $T^\phi$, such that $\Psi(g)=g$, $\Theta(h)=h$ and $\mathrm{Ax}(g) \cap \mathrm{Ax}(h)$ contains an edge $e$.

There exists $N \in \NN^*$ such that $\Psi^{N}=\Theta^{N}$.
\end{lemma}

\begin{proof}
First note that $F_{\Psi}$ (resp. $F_\Theta$) fixes pointwise the axis of $g$ (resp. $h$). Therefore, both $\Psi$ and $\Theta$ preserves the stabilisers of the endpoints of $e$. By construction of $T^\phi$, one of the endpoints $v$ of $e$ is either a rigid or a QH with fibre vertex. Moreover, both $\Psi$ and $\Theta$ preserve the subgroup associated with an edge adjacent to $v$ distinct from $e$.

Therefore, we can apply Lemma~\ref{Lemautomorphismfixingadjacentedges}: there exists $N \in \NN^*$ such that both $\Psi^N$ and $\Theta^N$ act as the identity on the nonelementary subgroup $G_v$. By \Cref{Lem:equalityautofixnonelementary}, up to taking powers of $\Psi$ and $\Theta$, we have $\Psi^N=\Theta^N$. 
\end{proof}

\begin{lemma}\label{Lemma:Perfinite}
Let $G$ be a hyperbolic group relative to $\calp$. Let $\calh_0$ be a finite set of conjugacy classes of finitely generated subgroups of $G$ such that G is one-ended relative to $\calp \cup \calh_0$. Let $\phi=[\Phi] \in \Out(G,\calp \cup \calh_0^{(t)})$. The set $\bigcup_{n \in \NN} \mathrm{Per}_{\mathrm{NP}}(\phi^n)$ is finite. Consequently, there exists $N\in \NN^*$ such that any subgroup whose conjugacy class is in $\bigcup_{n \in \NN^*} \mathrm{Per}_{\mathrm{NP}}(\phi^n)$ belongs to a subgroup whose conjugacy class is in $\mathrm{Per}_{\mathrm{NP}}(\phi^N)$.
\end{lemma}

\begin{proof}
Note that, for every $N \in \NN^*$ and every $\Psi \in \phi$, we have $\mathrm{Per}(\Psi^N)=\mathrm{Per}(\Psi)$. Thus, we will generally take a power of the considered automorphisms if needed.

Let $T^\phi$ be the above described tree associated with $\phi$. Up to taking a power of $\phi$, we may suppose that $\phi$ acts trivially on $G \backslash T^\phi$, that the homomorphism $\langle \phi \rangle \to \Out(G_e)$ is trivial for every $e \in ET^\phi$ and that the homomorphism $\langle \phi \rangle \to \Out(G_v)$ is trivial for every rigid vertex $v$.

\textbf{Counting elliptic subgroups:} Let $[H]\in \bigcup_{n \in \NN^*} \mathrm{Per}_{\mathrm{NP}}(\phi^n)$. Suppose first that every element of $H$ is elliptic in $T^\phi$. By \Cref{lem:elliptic_subgp_or_loxodromic_elet}, $H$ is elliptic in $T^\phi$. Since $[H] \in \bigcup_{n \in \NN^*} \mathrm{Per}_{\mathrm{NP}}(\phi^n)$, the group $H$ is not contained in the stabiliser of a vertex which is elementary. By \Cref{Coro:stabedgeinperiodicsubgroup}, it is also not contained in the stabiliser of a QH with finite fibre vertex. Therefore, the group $H$ is contained in the stabiliser of a vertex $v$ which is rigid. By \Cref{lem:additionalpropTphi}, we have in fact $H=G_v$.

In particular, the stabiliser of any vertex of $T^\phi$ contains at most one conjugacy class of elliptic subgroups in $\bigcup_{n \in \NN^*} \mathrm{Per}_{\mathrm{NP}}(\phi^n)$. Since the action of $G$ on $T^\phi$ has finitely many orbits of vertices, the set $\bigcup_{n \in \NN} \mathrm{Per}_{\mathrm{NP}}(\phi^n)$ contains only finitely many conjugacy classes of elliptic subgroups of $T^\phi$.

\textbf{Counting subgroups containing loxodromic elements:} Suppose now that $H$ contains a loxodromic element $h$. Let $n \in \NN^*$, let $\Psi \in \phi^n$ be such that $\mathrm{Per}(\Psi)=H$ and let $N \in \NN^*$ be such that $\Psi^N(h)=h$. By Remark~\ref{Remark:Phielliptic}, the isometry $F_{\Psi^N}$ is elliptic. Let $D$ be a finite fundamental domain for the action of $G$ on $T^\phi$. Up to taking a conjugate of $H$, we may suppose that $D$ contains an edge $e$ of the axis of $h$. 

Let $\Theta \in \phi^n$ with $[\mathrm{Per}(\Theta)] \in \mathrm{Per}_{\mathrm{NP}}(\phi^n)$ and let $h' \in G$ loxodromic be such that $\Theta^N(h')=h'$ and that $\mathrm{Ax}(h')$ contains $e$. As above the isometry $F_{\Theta^N}$ is elliptic in $T^\phi$. By Lemma~\ref{Lem:equalityautomorphismfixingaxis}, there exists $m \in \NN^*$ such that $\Psi^m=\Theta^m$. Therefore, we see that $\mathrm{Per}(\Psi)=\mathrm{Per}(\Theta)$. Hence the conjugacy class of $H$ in $\mathrm{Per}_{\mathrm{NP}}(\phi^n)$ is entirely determined by the edges of the fundamental domain $D$ contained in translates of axes of elements of $H$. Since $D$ is finite, and since the natural map $\mathrm{Per}_{\mathrm{NP}}(\phi^{n}) \to \mathrm{Per}_{\mathrm{NP}}(\phi^{(n+1)!})$ is injective, there exist only finitely many $[H] \in \bigcup_{n \in \NN^*} \mathrm{Per}_{\mathrm{NP}}(\phi^n)$ containing the conjugacy class of a loxodromic element. (In fact, their number is bounded above by the number of edges in the fundamental domain.)

As we have ruled out every case, we see that the set $\bigcup_{n \in \NN^*} \mathrm{Per}_{\mathrm{NP}}(\phi^n)$ is finite. 

The second assertion follows from the first since the set $\bigcup_{n \in \NN^*} \mathrm{Per}_{\mathrm{NP}}(\phi^{n!})$ is a nondecreasing sequence of sets exhausting $\bigcup_{n \in \NN^*} \mathrm{Per}_{\mathrm{NP}}(\phi^{n})$.
\end{proof} 

\begin{defn}
    Let $\phi \in \Out(G,\calp \cup \calh_0^{(t)})$ and let $N \in \NN^*$ be the integer given by \Cref{Lemma:Perfinite}. If $N=1$, we say that $\phi$ is \emph{almost rotationless}.
\end{defn}

Note that \Cref{Lemma:Perfinite} implies that every element of $\Out(G,\calp \cup \calh_0^{(t)})$ has an almost rotationless power.

\Cref{Lemma:Perfinite} has also the following corollary regarding \emph{isogredience classes} of automorphisms. Recall that two automorphisms $\Phi,\Psi$ of a group $G$ are in the same isogredience class if there exists $g \in G$ such that $\Phi=\mathrm{ad}_g\Psi\mathrm{ad}_g^{-1}$. The isogredience class of $\Phi$ is contained in its outer class. \Cref{Lemma:Perfinite} immediately implies the following.

\begin{corollary}
    Let $G$ be a hyperbolic group relative to $\calp$. Let $\calh_0$ be a finite set of conjugacy classes of finitely generated subgroups of $G$ such that G is one-ended relative to $\calp \cup \calh_0$. Let $\phi=[\Phi] \in \Out(G,\calp \cup \calh_0^{(t)})$. 
    
    There exist only finitely many isogredience classes of automorphisms in $\phi$ with a nonelementary periodic subgroup.
\end{corollary}

Notice that \Cref{Lemma:Perfinite} implies that there exist only finitely many isomorphism classes of nonelementary periodic subgroups for automorphisms contained in a given outer class. We ask for the following generalisation when the peripherals are virtually polycyclic.

\begin{question}
    Let $G$ be a hyperbolic group relative to virtually polycyclic groups. Do there exist only finitely many isomorphism classes of nonelementary periodic subgroups for automorphisms of $G$?
\end{question}

The case where the peripherals are abelian and $G$ is torsion free follows from the work of Guirardel--Levitt~\cite{guirardel2016mccool,guirardel2016vertex}.

\subsection{The Periodic JSJ tree}
We now need to understand vertex stabilisers of a JSJ tree given by \cref{Thm:existenceJSJ} and its acylindricity in order to apply Theorem~\ref{Thm:FJCacylindrical}. We set $\calh=\bigcup_{n \in \NN^*}\mathrm{Per}_{\mathrm{NP}}(\phi^n)$, and refer to the JSJ tree over elementary subgroups relative to $\calp \cup \calh_0 \cup \calh$ as $\Tper$. (Again, while this notation is neither standard nor entirely unambiguous, we use it consistently through our proofs.)

We highlight the fact that, if the subgroups in $\calh$ are finitely generated, then the acylindricity and the understanding of the vertex stabilisers of $\Tper$ mostly follow from \Cref{Thm:existenceJSJ}~$(10)$ (using our \cref{Lemma:Perfinite} for the other assumption). The main technical difficulty is thus to show that, when the periodic subgroups are not finitely generated, we still have a complete understanding of the stabilisers of rigid vertices.

\begin{remark}\label{rmk:AlmostRotationlessJSJ}
 Note that, by Lemma~\ref{Lemma:Perfinite}, there exists an almost rotationless power $\phi^N$ of $\phi$ such that every subgroup of $G$ whose conjugacy class is in $\calh$ is contained in a subgroup whose conjugacy class is in $\mathrm{Per}_{\mathrm{NP}}(\Phi^N)$. Thus, every $(\cala,\calp \cup \calh_0 \cup \calh)$-tree is an $(\cala,\calp \cup \calh_0 \cup \mathrm{Per}_{\mathrm{NP}}(\Phi^N))$-tree, and conversely. Thus the JSJ tree of $G$ relative to $\calp \cup \calh_0 \cup \calh$ is also the JSJ tree relative to $\calp \cup \calh_0 \cup \mathrm{Per}_{\mathrm{NP}}(\Phi^N)$ (see~\cite[Definition~2.12]{guirardellevitt2017jsj}). Therefore, we only need to work with almost rotationless automorphisms and we will still get results regarding the periodic JSJ tree associated with an arbitrary automorphism.
\end{remark}

Let $\phi^N$ be an almost rotationless power of $\phi$. Let $\mathrm{Per}_{\mathrm{NP}}(\phi^N)=\{[H_1],\ldots,[H_k]\}$ where, for every $i \in \{1,\ldots,k\}$, the group $H_i$ is not elementary and there exist $\Phi_i \in \phi^N$ such that $H_i=\mathrm{Per}(\Phi_i)$. Note that, for every $i \in \{1,\ldots,k\}$, \cref{Thm:existenceJSJ}~(\ref{thmJSJ:subgroups_in_H_are_rigid}) gives that the group $H_i$ fixes a unique rigid vertex $v_{i}$ in $\Tper$ since $H_i$ is nonelementary. 

Recall the construction of $T^\phi$ at the beginning of \Cref{section:constructiontrees}. Note that the trees $T^\phi$ and $\Tper$ are compatible. Let $\Tref$ be the unique minimal tree  which refines $T^\phi$ and $\Tper$. We denote by $p_\phi \colon \Tref \to T^\phi$ and $p_{\Per} \colon \Tref \to \Tper$ the associated $G$-equivariant projections. The tree $\Tref$ satisfies the following properties: a subgroup $H$ of $G$ stabilises a point in $\Tref$ if and only if $H$ stabilises a point in both $T^\phi$ and $\Tper$. Moreover, for every edge $e \in E\Tref$, at least one of the images $p_\phi(e)$ or $p_{\Per}(e)$ is not reduced to a point. 

Note that, since the actions of $G$ on $T^\phi$ and $\Tper$ are acylindrical, by minimality of $\Tref$, the action of $G$ on $\Tref$ is also acylindrical. (Any sufficiently long path in $\Tref$ will project to a path of length at least $3$ in at least one of $T^\phi$ and $\Tref$; edge stabilisers are not changed by the projection map, and so the stabiliser of the path must have been finite to begin with.) Moreover, by uniqueness of $\Tref$, and since the outer automorphism $\phi$ preserves $T^\phi$ and $\Tper$, we see that $\phi$ also preserves $\Tref$.

For $i \in \{1,\ldots,k\}$, let $F_{\Phi_i}$ be the isometry of $\Tref$ induced by $\Phi_i$. As in Remark~\ref{Remark:Phielliptic}, one can show that $F_{\Phi_i}$ is elliptic in $\Tref$ (this uses the acylindricity of the action of $G$ on $\Tref$). For every $i \in \{1,\ldots,k\}$, let $\Tref_{H_i}$ be the minimal tree of $H_i$ in $\Tref$. It might be that every element of some $H_i$ is elliptic in $\Tref$: in this case \Cref{lem:elliptic_subgp_or_loxodromic_elet} implies that the whole subgroup $H_i$ is also elliptic. Then $H_i$ stabilises a vertex in both $\Tper$ and $T^\phi$; since $H_i$ is non-elementary these vertices are unique. But since every edge of $\Tref$ survives in the projection to at least one of $T^\phi$ and $\Tper$, there cannot be an edge of $\Tref$ stabilised by $H_i$, and we may take the unique fixed vertex as the minimal invariant tree in this case.

\begin{lemma}\label{Lem:everypointfixedminimaltree}
Suppose that $\phi$ is almost rotationless. Let $\Phi \in \mathrm{NP}(\phi)$, $H = \Per(\Phi)$, $v$ a vertex in its minimal invariant tree $\Tref_H$, 
and $F_\Phi$ the isometry of $\Tref$ induced by $\Phi$. There exists $n \in \NN^*$ such that $F_{\Phi^n}$ fixes v.
\end{lemma}

\begin{proof}
If $\Tref_H$ consists of a single vertex $v$, then this is the unique vertex stabilised by $H$. Since $\Phi$ preserves $H$, $F_\Phi$ must also fix $v$. Otherwise, the vertex $v$ is contained in the axis of some $g \in H$. In particular, since $g$ is periodic, there exists $n \in \NN^*$ with $\Phi^n(g)=g$, and hence $F_{\Phi^n}$ preserves this axis. Since $F_{\Phi}$ is elliptic in $\Tref$, the isometry $F_{\Phi^n}$ fixes elementwise the axis of $g$. In particular, $F_{\Phi^n}$ fixes the vertex $v$.
\end{proof}

\begin{prop}
    \label{Lem:intersectionminimaltreestrivial}
Let $G$ be a hyperbolic group relative to $\calp$. Let $\calh_0$ be a finite set of conjugacy classes of finitely generated subgroups of $G$ such that G is one-ended relative to $\calp \cup \calh_0$. Let $\Phi$ and $\Psi$ be two representatives of an almost rotationless $\phi \in \Out(G,\calp \cup \calh_0^{(t)})$. 
Let $H=\Per(\Phi)$ and $K=\Per(\Psi)$ be two non-elementary periodic subgroups of $\phi$, perhaps conjugate. Then their minimal invariant trees $\Tref_H$ and $\Tref_K$ have non-empty intersection if and only if $H=K$.
\end{prop}

\begin{proof}
If $H=K$ then the minimal invariant trees $\Tref_H$ and $\Tref_K$ are equal, so only one direction needs proof. Consider the minimal invariant trees of $H$ and $K$ in $T^\phi$. If $T^\phi_H$ and $T^\phi_K$ do not intersect, then neither do $\Tref_H$ and $\Tref_K$, so for the remainder of the proof we assume there is an intersection here. If the intersection contains an edge, \cref{Lem:equalityautomorphismfixingaxis} implies that $H=K$, so from now on assume the intersection is a single vertex $v=T^\phi_H \cap T^\phi_K$. 

Each of $H$ and $K$ stabilise a unique rigid vertex in $\Tper$, and whenever $H \neq K$ we will construct an $(\cala, \calp \cup \calh_0 \cup \per)$-tree where $H$ and $K$ stabilise different vertices. This prevents $\langle H, K \rangle$ being contained in a rigid vertex group, by \cref{Thm:existenceJSJ}~(\ref{thmJSJ:subgroups_in_H_are_rigid}). But then the rigid vertices stabilised by $H$ and $K$ are distinct, and the trees in $\Tref$ (containing $\Tref_H$ and $\Tref_K$ as subtrees) collapsing to them must be disjoint.

We distinguish two cases, according to the nature of $v$. Note that $v$ is not QH with fibre by \Cref{coro:noQHwithfibrevertex}.

\textbf{Case 1: $v$ is a rigid vertex.} 
By~\Cref{lem:additionalpropTphi} if $H$ or $K$ is elliptic and~\Cref{Coro:stabedgeinperiodicsubgroup} otherwise, there exists $N \in \NN^*$ such that $G_v \subseteq \Fix(\Phi^N)$ and $G_v \subseteq \Fix(\Psi^N)$. 

Hence both $\Psi^N$ and $\Phi^N$ fix elementwise the same nonelementary subgroup. By \Cref{Lem:equalityautofixnonelementary}, we see that $H=K$.

\textbf{Case 2: $v$ is elementary.} Let $\calt_v$ be the set of minimal trees $T_K^\phi$ with $[K] \in \mathrm{Per}_{\mathrm{NP}}(\phi)$ which contain $v$. Since, for every $[K] \in \mathrm{Per}_{\mathrm{NP}}(\phi)$, the group $K$ is nonelementary, no tree $S \in \calt_v$ is reduced to a point. By Lemma~\ref{Lem:equalityautomorphismfixingaxis}, for all distinct $S,S' \in \calt_v$, the intersection $S \cap S'$ is reduced to $v$. Let $E(v)$ be the set of edges adjacent to $v$. We have a partition $$E(v)=E \coprod (E_S)_{S \in\calt_v} $$ where for every $e \in E$, the edge $e$ is not contained in any $S \in \calt_v$ and, for every $S \in \calt_v$ and every $e \in E_S$, the edge $e$ is contained in $S$. 

Let $S_v$ be the tree with one central vertex $v_0$ adjacent to all the other vertices and that the leaves $v_S$ are indexed by the trees $S \in \calt_v$. We suppose that the stabiliser of $v_0$ is equal to $G_v$ and, for every $S \in \calt_v$, that the stabiliser of $v_S$ is equal to the setwise stabiliser $\Stab(E_S)$ of $E_S$. Let $T'$ be the tree obtained from $T$ by blowing up $S_v$ at $v$ and attaching for every $e \in E$, the edge $e$ to $v_0$ and for every $S \in \calt_v$ and every $e \in E_S$, by attaching the edge $e$ to $v_S$. 

Note that the tree $T'$ obtained is an $(\cala,\calp \cup \calh_0)$-tree. Moreover, for any $[K_1],[K_2] \in \mathrm{Per}_{\mathrm{NP}}(\phi)$ with $T_{K_1}^\phi \neq T_{K_2}^\phi \in \calt_v$ the minimal trees of $T_{K_1}'$ and $T_{K_2}'$ of $K_1$ and $K_2$ in $T'$ are disjoint.

Let $U$ be the $(\cala,\calp \cup \calh_0 \cup \calh)$-tree obtained from $T'$ by collapsing the minimal tree of every $H$ with $[H] \in \mathrm{Per}_{\mathrm{NP}}(\phi)$. Then $K$ and $H$ fix distinct points in $U$. Thus, $U$ is an $(\cala,\calp \cup \calh_0 \cup \calh)$-tree where $H$ and $K$ fix distinct points. By Theorem~\ref{Thm:existenceJSJ}~$(7)$, the groups $H$ and $K$ fix distinct rigid vertices in $\Tper$. Therefore, the minimal trees $T_H^\phi$ and $T_K^\phi$ of $H$ and $K$ in $\Tref$ are disjoint.
\end{proof}

\begin{lemma}\label{Lem:preimagevertices}
Suppose $\phi$ is almost rotationless. For every $H \in \per$, the stabiliser of the vertex $v_H$ of $\Tper$ is equal to the global stabiliser $G_{\Tref_H}$ of the minimal tree $\Tref_H$ of $H$ in $\Tref$.
\end{lemma}

\begin{proof}
First note that, since the projection $p_{\Per} \colon \Tref \to \Tper$ is $G$-equivariant, the tree $\Tref_H$ collapses onto $v_H$. Thus, we have $G_{\Tref_H} \subseteq G_{v_H}$.

Conversely, let $U$ be the $(\cala,\calp \cup \calh_0 \cup \calh)$-tree obtained from $\Tref$ by collapsing, for every $H \in \per$, the tree $\Tref_H$. For every $H \in \per$, let $w_H$ be the vertex of $U$ fixed by $H$. Note that, by Lemma~\ref{Lem:intersectionminimaltreestrivial}, for any $H,K \in \per$ and every $g \in G$ such that $gHg^{-1} \neq K$, the trees $g\Tref_H$ and $\Tref_K$ are disjoint. Thus, the stabiliser of $w_H$ is equal to $G_{\Tref_H}$. Since the vertex $v_H$ of $\Tper$ is rigid in every $(\cala,\calp \cup \calh_0 \cup \calh)$-tree (see~Theorem~\ref{Thm:existenceJSJ}~$(7)$), the group $G_{v_H}$ is elliptic in $U$ and contains $H$. As $H$ fixes a unique point in $U$, which is $w_H$, we see that $G_{v_H} \subseteq G_{\Tref_H}$.
\end{proof}

\begin{lemma}\label{Lem:Rigidverticesareperiodic}
Suppose that $\phi$ is almost rotationless. For every $H \in \per$, and the unique vertex $v_H \in V\Tper$ it stabilises, we have $G_{v_H}=H$.
\end{lemma}

\begin{proof}
Let $H \in \per$, let $\Phi \in \phi$ be the automorphism satisfying $\Per(\Phi)=H$, and $F_\Phi$ the associated isometry of $\Tref$. By Lemma~\ref{Lem:preimagevertices}, it suffices to show that the stabiliser $G_{\Tref_H}$ of $\Tref_H$ is equal to $H$. By~\Cref{Lem:intersectionminimaltreestrivial} and the fact that $g\Tref_H=\Tref_{gHg^{-1}}$ for every $g \in G$, the group $G_{\Tref_H}$ is equal to the normaliser $N(H)$ of $H$ in $G$. 

Suppose first that $H$ fixes a (necessarily unique) vertex $v \in V\Tref$. Then $v$ has nonelementary stabiliser. By \Cref{lem:additionalpropTphi}, the point $p_\phi(v)$ is the unique rigid vertex fixed by $H$ in $T^\phi$. Thus, $N(H)$ fixes $p_\phi(v)$. By~\Cref{lem:additionalpropTphi}, we have $N(H) \subseteq G_{p_\phi(v)} \subseteq H$ and $H=N(H)$.

Suppose now that $H$ does not fix a vertex. By \Cref{lem:elliptic_subgp_or_loxodromic_elet}, $H$ contains a loxodromic element, and so
$\Tref_H$ is an infinite subtree of $\Tref$. Let $g \in N(H)$. 
Let $v_1$ and $v_2$ be vertices in $\Tref_H$, at distance greater than the acylindricity constant $K$. Observe that the acylindricity assumption implies that there are only finitely many elements $g'$ with the property that $gv_i= g'v_i$ for $i=1,2$ (since they will be contained in the coset $g(G_{v_1} \cap G_{v_2})$, which is finite). Call this set $F$.

By \Cref{Lem:everypointfixedminimaltree}, up to taking a power of $\Phi$, the isometry $F_\Phi$ fixes $v_1,gv_1,v_2$ and $gv_2$.  
Thus, for $i=1,2$ we have that $\Phi$ preserves the set of elements sending $v_i$ to $gv_i$, and in particular their finite intersection $F$. 
As $F$ is finite and contains $g$, a power of $\Phi$ fixes $g$. This shows that $g \in H = \Per(\Phi)$ (recalling that the periodic subgroup is invariant under taking powers) and $N(H)=H$. \qedhere

\end{proof}

\begin{thm}\label{Coro:rigidvertexperiodic}
    Let $G$ be a hyperbolic group relative to $\calp$. Let $\calh_0$ be a finite set of conjugacy classes of finitely generated subgroups of $G$ such that G is one-ended relative to $\calp \cup \calh_0$. Let $\phi=[\Phi] \in \Out(G,\calp \cup \calh_0^{(t)})$. Let $\phi^N$ be an almost rotationless power of $\phi$.   

    For every rigid vertex $v\in V\Tper$, there exists $[H]\in \mathrm{Per}_{\mathrm{NP}}(\phi^N)$ such that $G_v=H$. 
    Conversely, for every $[H]\in \mathrm{Per}_{\mathrm{NP}}(\phi^N)$, there exists a rigid vertex $v\in V\Tper$ with $G_v=H$.
\end{thm}

\begin{proof}
    Up to taking a power of $\phi$, we may assume that $\phi$ is almost rotationless (see~\Cref{rmk:AlmostRotationlessJSJ}). Let $v \in V\Tper$ be rigid. We claim that there exists $[H] \in \per$ with $H \leq G_v$. Consider the action of $G_v$ in $\Tref$. If $G_v$ fixes a vertex in $\Tref$, then $G_v$ also stabilises a non-elementary vertex $w$ of $T^\phi$. If $w$ is rigid, then $\langle \phi \rangle^0 \to \Out(G_w)$ has finite image, so $G_v$ is periodic for some $\phi^k$, and so for $\phi$ (since we assume it is almost rotationless).

    Otherwise $w$ is a QH with fibre vertex. Note that all elementary subgroups of $G_w$ are virtually cyclic. In particular, the edges adjacent to $v \in V\Tper$ are all virtually cyclic, hence finitely generated. Thus, $G_v$ is a finitely generated subgroup of $G_w$. Since $w$ is a QH with fibre vertex, it is locally quasi-convex. Thus, $G_v$ is hyperbolic. 

    Note that, since $G_v$ is hyperbolic and a subgroup of $\Stab(w)$, the restriction $\calp\cup\calh_0 \cup \mathrm{Per}_{\mathrm{NP}}(\phi)|_{G_v}$ of $\calp\cup\calh_0 \cup \mathrm{Per}_{\mathrm{NP}}(\phi)$ to $G_v$ is a finite family of virtually cyclic subgroups of $G_v$ by \Cref{lem:orbifold_periodic_is_fg}. Thus, we can apply~\cite[Theorem~3.9]{GuirardelLevitt2015} to show that  $\Out(G_v,\calp\cup\calh_0^{(t)} \cup \mathrm{Per}_{\mathrm{NP}}(\phi)|_{G_v} \cup \mathrm{Inc}_v)$ (which has $\Out(G_v,\calp\cup(\calh_0 \cup \mathrm{Per}_{\mathrm{NP}}(\phi)|_{G_v}\cup \mathrm{Inc}_v)^{(t)})$ as a finite index subgroup) is infinite if and only if $G_v$ has an  $(\cala,\calp\cup \calh_0 \cup \mathrm{Per}_{\mathrm{NP}}(\phi)|_{G_v} \cup \mathrm{Inc}_v)$-splitting. Since $v$ is rigid, no such splitting of $G_v$ exists. Therefore, $\Out(G_v,\calp\cup \calh_0^{(t)} \cup \mathrm{Per}_{\mathrm{NP}}(\phi)|_{G_v}\cup \mathrm{Inc}_v)$ is finite. In particular, the group $G_v$ is a nonelementary periodic subgroup of some power of $\phi$.   
    
    Now suppose that $G_v$ is not elliptic in $\Tref$. By \Cref{lem:elliptic_subgp_or_loxodromic_elet}, $G_v$ contains a loxodromic element in $\Tref$ and the minimal invariant tree $\Tref_{G_v} \subseteq p_{\Per}^{-1}(v)$ of $G_v$ in $\Tref$ is infinite. If this is contained in $\Tref_H$ for some $H \in \per$, then $v$ is the unique vertex of $\Tper$ stabilised by $H$, and by \cref{Lem:Rigidverticesareperiodic} $G_v$ (as the stabiliser of this vertex) is equal to $H$.

    So suppose $\Tref_{G_v}$ is not contained in any $\Tref_H$. Since the $\Tref_H$ are disjoint by \Cref{Lem:intersectionminimaltreestrivial}, there is some edge of $\Tref_{G_v}$ contained in no $\Tref_H$, and hence collapsing all the $\Tref_H$ will give an $(\cala, \calp \cup \calh_0 \cup \per)$-tree where $G_v$ is not elliptic, contradicting \cref{Thm:existenceJSJ}~(7).

    Conversely, let $[H]\in \mathrm{Per}_{\mathrm{NP}}(\phi)$. By construction of $\Tper$, the group $H$ fixes a vertex $v$ of $\Tper$. Since $H$ is nonelementary, such a vertex is unique and is rigid by \Cref{Thm:existenceJSJ}~$(10)$. By Lemma~\ref{Lem:Rigidverticesareperiodic}, we have $G_v=H$.
\end{proof}

\begin{corollary}\label{Coro:edgestabelementary}
Let $G$ be a hyperbolic group relative to $\calp$. Let $\calh_0$ be a finite set of conjugacy classes of finitely generated subgroups of $G$ such that G is one-ended relative to $\calp \cup \calh_0$. Let $\phi=[\Phi] \in \Out(G,\calp \cup \calh_0^{(t)})$. Let $\calh=\bigcup_{n\in \NN^*} \mathrm{Per}_{\mathrm{NP}}(\phi^n)$.

\begin{enumerate}
    \item For every geodesic edge path $\gamma$ of $\Tper$ of length $3$ and every automorphism $\Psi \in \phi$ preserving $\gamma$, there exist a vertex $v$ of $\gamma$ and $g \in G_v$ of infinite order fixed by a power of $\Psi$.
    \item The group $G_\Phi$ acts acylindrically on $\Tper$.
\end{enumerate}

\end{corollary}

\begin{proof}
Let $\gamma$ be a geodesic edge path of length $3$ in $\Tper$ preserved by an automorphism $\Psi \in \phi$. Suppose that there exists an edge $e$ of $\gamma$ whose stabiliser is virtually cyclic (this applies in particular when one of the vertices of $\gamma$ is QH with fibre).
Hence $\Out(G_e)$ is finite. Thus, $\Psi$ has a power acting as the identity on an infinite cyclic subgroup of $G_e$.

Thus, we may suppose that $\gamma$ only contains elementary and rigid vertices. Since $\gamma$ has length $3$ and since $\Tper$ is bipartite, it contains an interior vertex $v$ which is rigid. Let $e_1,e_2$ be the two edges of $\gamma$ adjacent to $v$. Then $G_{e_1}\cap G_{e_2}$ is finite by \Cref{Thm:existenceJSJ}~$(8)$ and for every $i \in \{1,2\}$ the group $G_{e_i}$ is its own normaliser in $G_v$. 

Let $\phi^N$ be an almost rotationless power of $\phi$. By~\Cref{Coro:rigidvertexperiodic}, there exists $[H]\in \mathrm{Per}_{\mathrm{NP}}(\phi^N)$ such that $G_v=H$. Let $\Phi_H \in \phi^N$ be such that $\mathrm{Per}(\Phi_H)=H$. Since $\Psi$ preserves $H$, and since $H$ is its own normaliser by \Cref{Lem:Rigidverticesareperiodic}, there exists $h \in H$ such that $\Psi^N=\mathrm{ad}_h\Phi_H$. As $\Psi^N$ preserves both $G_{e_1}$ and $G_{e_2}$, we see that $h\in N(G_{e_1}) \cap N(G_{e_2})$, which is finite as explained above. Up to taking a power of $\Phi_H$ and $\Psi^N$, we may assume that both $\Phi_H$ and $\Psi^N$ act trivially on $N(G_{e_1}) \cap N(G_{e_2})$. Taking further powers of $\Psi^N$ and $\Phi_H$ shows that there exists $M \in \NN$ such that $\Psi^M=\Phi_H^M$. 

Let $g \in H$ be of infinite order, which exists since $H$ is nonelementary. Then $\Psi$ has a power which fixes $g$. This proves Assertion~$(1)$.

We now prove that the action of $G_\Phi$ on $\Tper$ is acylindrical. Note that, if a finite index subgroup of $G_{\Phi}$ acts acylindrically on $\Tper$, so does $G_{\Phi}$. Thus, we may assume that $\phi$ is almost rotationless. 

By Theorem~\ref{Thm:existenceJSJ}~$(2)$, the action of $G$ on $\Tper$ is acylindrical. Thus, by Lemma~\ref{lem:PromoteAcylindricalactionsv2}~$(2)$ (which we can apply by Assertion~$(1)$), it suffices to prove that, for every $n \in \NN^*$, every $\Psi \in \phi^n$ and every $g \in \Fix(\Psi)$, the element $g$ is elliptic in $T$. 


Let $g$ be as above. If $g$ is peripheral, then $g$ fixes a point by construction of $\Tper$. 

Suppose now that $g$ is nonperipheral. Suppose towards a contradiction that $g$ is loxodromic in $\Tper$. Since $\Psi(g)=g$, the characteristic set of the isometry $F_{\Psi}$ contains the axis of $g$. By \Cref{lem:element_fixing_axis}, up to taking a power of $\phi$ and changing the representative $\Psi$, the isometry $F_{\Psi}$ fixes pointwise the axis of $g$. By Lemma~\ref{lem:PromoteAcylindricalactionsv2}~$(1)$, up to taking a power of $\Psi$, the automorphism $\Psi$ fixes elementwise a nonabelian free group of loxodromic elements. In particular, since every peripheral element fixes a point in $\Tper$, we see that $\Psi$ fixes a nonabelian free group of nonperipheral elements, hence a nonelementary subgroup. Since $\phi$ is almost rotationless, for every $n \geq 1$, we have $\mathrm{Per}_{\mathrm{NP}}(\phi)=\mathrm{Per}_{\mathrm{NP}}(\phi^{n})$, so that $[\mathrm{Per}(\Psi)]\in \per$. Thus, there exists $[H] \in \mathrm{Per}_{\mathrm{NP}}(\phi)$ such that $g \in \Fix(\Psi) \subseteq H$. In that case, the element $g$ is elliptic in $\Tper$ by construction of $\Tper$, a contradiction. 

Therefore, the element $g$ is elliptic in $\Tper$ and we can apply Lemma~\ref{lem:PromoteAcylindricalactionsv2}~$(2)$ to prove that the action of $G_\Phi$ on $\Tper$ is acylindrical.
\end{proof}

Let $G$ be a finitely generated group and let $ \Phi \in \Aut(G)$. Suppose that $\Phi$ has a power which preserves the conjugacy class of a malnormal subgroup $F$ of $G$. We then denote by $F_\Phi$ the group $F \rtimes_{\mathrm{ad}_g\circ \Phi^{n_F}}\ZZ$, where $n_F$ is the minimal positive integer such that $\Phi^{n_F}$ preserves the conjugacy class of $F$ and $g\in G$ is such that $\mathrm{ad}_g\circ \Phi^{n_F}(F)=F$. Since $F$ is malnormal, the group $F_\Phi$ does not depend on $g$. Note that the group $F_\Phi$ only depends on the outer class of $\Phi$.

\begin{corollary}\label{Cor:OneEndedFJCrelhyp}
Let $G$ be a hyperbolic group relative to $\calp$. Let $\calh_0$ be a finite set of conjugacy classes of finitely generated subgroups of $G$ such that G is one-ended relative to $\calp \cup \calh_0$. Let $\phi=[\Phi] \in \Out(G,\calp \cup \calh_0^{(t)})$. 

If for every $[P] \in \calp$ the group $P_\Phi$ is in $\FJCX$, then $G_\Phi$ is in $\FJCX$.
\end{corollary}

\begin{proof}
Consider the action of $G_\Phi$ on $\Tper$. By Corollary~\ref{Coro:edgestabelementary}, the action of $G_\Phi$ on $\Tper$ is acylindrical. Up to taking a power of $\Phi$ (which does not change the conclusion by \Cref{Thm:FJCclosedsubgroups}), we may assume that $\phi=[\Phi]$ is almost rotationless.

By Theorem~\ref{Thm:FJCacylindrical}, it suffices to show that, for every $v \in V\Tper$, the stabiliser $(G_\Phi)_v$ of $v$ in $G_\Phi$ belongs to $\FJCX$. Note that, for every $v \in V\Tper$, the group $(G_\Phi)_v$ can be written as a semi direct product $G_v \rtimes \ZZ$, where $G_v$ is the stabiliser of $v$ in $G$.

Suppose first that $G_v$ is elementary. If $G_v \subseteq P$ for some $[P] \in \calp$, then $(G_\Phi)_v$ is a subgroup of $P_\Phi$. In particular, it belongs to $\FJCX$ by Theorem~\ref{Thm:FJCclosedsubgroups}~$(1)$. If $G_v$ is infinite virtually cyclic, then $(G_\Phi)_v$ belongs to $\FJCX$. 

Suppose now that $G_v$ is QH with fibre. Then $G_v$ fits in a short exact sequence $$1 \to F \to G_v \to \pi_1(\Sigma_v) \to 1,$$ where $\Sigma_v$ is a hyperbolic $2$-orbifold and $F$ is finite. Moreover, up to taking a power of $\Phi$ (which is possible by Theorem~\ref{Thm:FJCclosedsubgroups}~$(2)$), since $\Out(G_v)$ has a finite index subgroup acting as the identity on $F$, there exists $\Psi \in \phi$ preserving $G_v$ and fixing $F$ elementwise. Thus, we have a short exact sequence $$1 \to F \to (G_\Phi)_v \to \pi_1(\Sigma_v) \rtimes \ZZ \to 1. $$  The groups $F$ and $\pi_1(\Sigma_v) \rtimes \ZZ$ belong to $\FJCX$.  Indeed, the Farrell--Jones Conjecture holds for all fundamental groups of (not necessarily compact) connected manifolds (possibly with boundary) of dimension at most $3$, by \cite[Theorem~16.1(i)(e)]{LuckBookProj}.  As explained in loc. cit. the result is deduced from Corollary 1.3 and Remark 1.4 of \cite{BartelsFarrellLuck2014} where \cite{Roushon2008a,Roushon2008b} is used. Moreover, for every virtually cyclic group $Q \subseteq \pi_1(\Sigma_v) \rtimes \ZZ$, the preimage of $Q$ in $(G_\Phi)_v$ is virtually cyclic, and hence belongs to $\FJCX$. Thus, the group $(G_\Phi)_v$ belongs to $\FJCX$.

Suppose that $G_v$ is rigid. By~\Cref{Coro:rigidvertexperiodic}, since $\phi$ is almost rotationless, there exists $[H] \in \mathrm{Per}_{\mathrm{NP}}(\phi)$ such that $G_v=H$. 

Suppose that $\Phi \in \phi$ is such that $\mathrm{Per}(\Phi)=H$. Then $(G_\Phi)_v$ is isomorphic to $\mathrm{Per}(\Phi) \rtimes_\Phi \ZZ$. By \Cref{Lem:PerFJC}, we have $(G_{\Phi})_v \in \FJCX$.

As we have ruled out every case, for every $v \in V\Tper$, the group $(G_\Phi)_v$ belongs to $\FJCX$. By Theorem~\ref{Thm:FJCacylindrical}, the group $G_\Phi$ belongs to $\FJCX$.
\end{proof}

\subsection{An aside on slender peripherals} \label{subsec:slender_peripherals}
We also isolate here an interesting consequence of \Cref{Coro:rigidvertexperiodic} for automorphisms of groups hyperbolic relative to slender groups. Recall that a group is \emph{slender} if all its subgroups are finitely generated.

\begin{thm}\label{Coro:Slenderperiodicisfixedoneended}
    Let $G$ be a hyperbolic group relative to a collection $\calp$ of slender groups and let $\Phi \in \Aut(G)$. There exists $N \in \NN^*$ such that $\mathrm{Per}(\Phi)=\Fix(\Phi^N)$ and $\mathrm{Per}(\Phi)$ is finitely generated.
\end{thm}

\begin{proof}
    See also the proof of~\cite[Theorem~8.2]{GuirardelLevitt2015}. We claim that it suffices to prove that $\mathrm{Per}(\Phi)$ is finitely generated. Indeed, suppose that $\mathrm{Per}(\Phi)$ is generated by $a_1,\ldots,a_n$. For every $i \in \{1,\ldots,n\}$, let $k_i$ be such that $\Phi^{k_i}(a_i)=a_i$. Let $N=k_1\ldots k_n$. Then $\mathrm{Per}(\Phi)=\Fix(\Phi^N)$. 
    
    So we prove that $\mathrm{Per}(\Phi)$ is finitely generated. Note that slender groups are NRH groups, so that $\Aut(G)=\Aut(G,\calp)$. Note also that, since $\calp$ is a set of conjugacy classes of slender groups, every elementary subgroup of $G$ is finitely generated. Thus, we may suppose that $\mathrm{Per}(\Phi)$ is not elementary.
    
    Let $\phi=[\Phi]\in \Out(G)$ and let $\calh=\mathrm{Per}_{\mathrm{NP}}(\phi)$. Let $(\calh_i)_{i \ge 1}$ be an increasing sequence of finite set of conjugacy classes of finitely generated subgroups of $G$ such that $\calh=\bigcup_{i \ge 1} \calh_i$. By~\cite{osin2006relatively}, for every $i \geq 1$, the group $G$ is finitely presented relative to $\calp \cup \calh_i$. Thus, the Stallings--Dunwoody JSJ deformation space $D_i$ of $G$ over finite groups relative to $\calp \cup \calh_i$ exists (see~\cite{guirardel2007deformation}). 
    
    We claim that there exists $i_0$ such that, for every $i\geq i_0$, we have $D_i=D_{i_0}$. The proof is identical to \cite[Lemma~8.2]{guirardellevitt2017jsj} replacing Linnell's accessibility Theorem by \cite[Proposition~4.2]{guirardel2007deformation}. We sketch the argument. Let $i \ge 1$. By~\cite[Lemma~2.15]{guirardellevitt2017jsj}, since edge groups are all finite and hence all universally elliptic, for every $T_i \in D_i$, there exists a sequence of refinements of trees $T_i^1 \to \ldots \to T_i^{i-1} \to T_i^i=T_i$ such that, for every $j\leq i$, we have $T_i^j \in D_j$. Let $j \le i$. Notice that, since edge groups are finite, $T_i^{j-1}$ and $T_i^j$ are not in the same deformation space if and only if they do not have the same infinite elliptic subgroups if and only if the number of orbits of vertices with infinite stabiliser is greater in $T_i^{j-1}$ than in $T_i^j$. By \cite[Proposition~4.2]{guirardel2007deformation}, the number of orbits of vertices of $T_i^1$ with infinite stabiliser is bounded by a constant $J$ depending only on $D_1$. Thus, the number of $j \le i$ such that $T_i^{j-1}$ and $T_i^j$ are not in the same deformation space is bounded by $J$. This forces the spaces $D_i$ to eventually stabilise, hence the claim. Let $D_{i_0}$ be the stabilised deformation space. As $\calp \cup \calh=\calp \cup \bigcup_{i \ge 1} \calh_i$, this shows that any tree in $D_{i_0}$ is a $(\calp \cup \calh)$-tree with finite edge groups. 
    
    Suppose first that $G$ is one-ended relative to $\calp \cup \calh$. Then $D_{i_0}$ consists of the trivial tree and $G$ also one-ended relative to $\calp \cup \calh_{i_0}$. Moreover, since $\calh_{i_0}$ is a finite set and all groups whose conjugacy class is in $\calh_{i_0}$ are finitely generated, up to taking a power of $\phi$, we have $\phi \in \Out(G,\calp \cup \calh_{i_0}^{(t)})$. Let $\Tper$ be the JSJ tree associated with $\calp \cup \calh_{i_0} \cup \calh$. 
    
     Since edge stabilisers of $\Tper$ are elementary, they are all finitely generated. Thus, every vertex stabiliser of $\Tper$ is also finitely generated.
    
    By \Cref{Coro:rigidvertexperiodic}, the group $\mathrm{Per}(\Phi)$ is equal to the stabiliser of a vertex of $\Tper$, hence is finitely generated.

    Suppose now that $G$ is not one-ended relative to $\calp \cup \calh$ and consider a tree $S$ in $D_{i_0}$. Recall that edge stabilisers in $S$ are all finite. Since $\mathrm{Per}(\Phi)$ is infinite, it fixes a unique vertex $v$. Since the deformation space of $S$ is invariant by $\Phi$, the group $\Phi(G_v)$ also fixes a unique vertex $w$ in $S$. As $\mathrm{Per}(\Phi) \subseteq \Phi(G_v)$ and as $\mathrm{Per}(\Phi)$ only fixes $v$, we see that $v=w$. This shows that $\Phi(G_v)=G_v$. 
    
    Note that, by minimality of $S$, the group $G_v$ is hyperbolic relative to the restriction $\calp_v$ of $\calp$ in $G_v$ and one-ended relative to $\calp_v \cup (\calh_{i_0})_v \cup \calh_v$. Since $\calp$ is a set of conjugacy classes of slender groups, so is $\calp_v$. Thus, the conclusion follows from the one-ended case applied to the restriction $\Phi|_{G_v}$. 
\end{proof}

We remark that Minasyan--Osin~\cite[Corollary~1.3]{minasyan2012fixed} also proved that the fixed subgroup of the automorphism of any hyperbolic group relative to slender groups is finitely generated.

\section{A combination theorem for the Farrell--Jones conjecture}
\label{sec:infinite_ended}

Let $G$ be a finitely generated group and let $\Phi \in \Aut(G)$. If $F$ is a malnormal subgroup of $G$ whose conjugacy class is $\Phi$-periodic, recall the definition of $F_\Phi$ from just above \Cref{Cor:OneEndedFJCrelhyp}. In this section, we prove the following combination theorem. 

\setcounter{thmx}{3}
\begin{thmx}\label{Thm:combinationthm}
    Let $G=G_1 \ast \ldots \ast G_k \ast F_N$ be a free product of finitely generated groups, let $\calf'=\{[G_1],\ldots,[G_k]\}$ and let $ \Phi \in \Aut(G,\calf')$. If for each $i \in \{1,\ldots,k\}$, the group $(G_i)_\Phi$ is in $\FJCX$, then $G \rtimes_\Phi \ZZ$ is in $\FJCX$.
\end{thmx}

The proof of Theorem~\ref{Thm:combinationthm} is by induction on $k+N$. Let $\calf' \leq \calf$ be a maximal proper $\Phi$-periodic free factor system. Up to taking a power of $\Phi$ (which does not change the conclusion of Theorem~\ref{Thm:combinationthm} by Theorem~\ref{Thm:FJCclosedsubgroups}~$(2)$), we may suppose that $\Phi\in \Aut(G,\calf)$. We will distinguish between two cases, according to whether $\calf$ is sporadic or not.

\subsection{The nonsporadic case}

\begin{lemma}\label{Lem:fullyirredcase}
    Let $G=G_1 \ast \ldots \ast G_k \ast F_N$ be a free product of groups, let $\calf=\{[G_1],\ldots,[G_k]\}$ and let $ \Phi \in \Aut(G,\calf)$ be fully irreducible. If for each $i \in \{1,\ldots,k\}$, the group $(G_i)_\Phi$ is in $\FJCX$, then $G \rtimes_\Phi \ZZ$ is in $\FJCX$.
\end{lemma}

\begin{proof}
    Let $S$ be the Grushko $(G,\calf)$-tree given by Lemma~\ref{Lem:existenceRtree}. Let $\calp_S(\Phi)$ be the $\lVert. \rVert_S$-maximal polynomial subgroups of $\Phi$. By Theorem~\ref{thm:relativehyperbolicityfullyirred}, up to taking a power of $\Phi$ (which does not change the conclusion of Lemma~\ref{Lem:fullyirredcase} by Theorem~\ref{Thm:FJCclosedsubgroups}~$(2)$) the group $G \rtimes_\Phi\ZZ$ is hyperbolic relative to the suspension of $\calp_S(\Phi)$. By Proposition~\ref{Prop:descriptionpolysubgroups}, for every $[P]\in \calp_S(\Phi)$, either $[P]\in \calf$ or $P$ is infinite cyclic. In either case $P_\Phi$ is contained in $\FJCX$. 
    By Theorem~\ref{Thm:FJCrelativelyhyp}, the group $G\rtimes_\Phi \ZZ$ is contained in $\FJCX$.
\end{proof}

\subsection{The sporadic case}
\label{subsec:sporadic_case}

This section follows~\cite[Proof of Proposition~4.1]{bestvina2023farrell}. Let $(G,\calf)$ be a sporadic free product and let $\Phi \in \Aut(G,\calf)$. Since $(G,\calf)$ is sporadic, the automorphism $\Phi$ induces a $G$-equivariant homeomorphism of the Bass--Serre tree $T_\calf$ associated with $\calf$. This induces an action of $G \rtimes_\Phi \ZZ$ on $T_\calf$. However, this action is not necessarily acylindrical. In order to apply Theorem~\ref{Thm:FJCacylindrical}, we will consider the action of $G \rtimes_\Phi \ZZ$ on the tree of cylinders of $T_\calf$ associated with an admissible relation that we now describe.

 Let $t$ be a generator of the $\ZZ$-factor. Up to taking a finite index subgroup of $G\rtimes_\Phi \ZZ$ (which is possible by Theorem~\ref{Thm:FJCclosedsubgroups}~$(2)$), we may suppose that $t$ fixes an edge $e$. In that case, edge stabilisers of the action of $G \rtimes_\Phi \ZZ$ are all infinite cyclic, generated by conjugates of $t$. Therefore, the commensurability relation is an admissible equivalence relation, and we define the tree of cylinders $T_c$ of $T_\calf$ relative to this admissible relation.

 \begin{lemma}\label{Lem:Stabcylinder}
     Let $Y$ be the cylinder of $T_\calf$ containing $e$. The stabiliser of $Y$ in $G\rtimes_\Phi \ZZ$ is isomorphic to $\langle \Fix(\Phi^n) \rangle_{n \ge 1} \rtimes_\Phi \ZZ$.
 \end{lemma}

\begin{proof}
    Note that any element $h \in G \rtimes_\Phi \ZZ$ can be written uniquely as $w^{-1}t^j$, where $w\in G$ and $j \in \ZZ$. Let $w^{-1}t^j \in \Stab(Y)$ with $w\in G$ and $j \in \ZZ$ and let $e'=w^{-1}t^je$. Then we have $G_{e'}=\langle w^{-1}tw \rangle$. Moreover, since $e,e' \in EY$, by definition of the commensurability relation, there exist $n,m \in \NN$ such that $$t^n=w^{-1}t^mw=w^{-1}\Phi^m(w)t^m.$$ In particular, we see that $n=m$ and $\Phi^m(w)=w$.
\end{proof}

\begin{lemma}\label{Lem:Actiontreecylacylindrical}
    The action of $G\rtimes_\Phi \ZZ$ on $T_c$ is acylindrical.
\end{lemma}

\begin{proof}
    The proof is identical to~\cite[Lemma~4.6]{bestvina2023farrell}. Let $v,v' \in VT_c$ with $d_{T_c}(v,v') \geq 6$. We may suppose that $v$ and $v'$ correspond to vertices $w$ and $w'$ in $T_\calf$ up to considering adjacent vertices in the path between them. Even after this operation, we have $d_{T_c}(v,v') \geq 4$. 

    Let $g\in G_v \cap G_{v'}$. Then $g$ fixes the path in $T_\calf$ between $w$ and $w'$. Since $d_{T_c}(v,v') \geq 4$, the path in $T_\calf$ between $w$ and $w'$ must contain two edges in distinct cylinders. Hence $g$ fixes two edges in distinct cylinders. Since edge stabilisers in $T_\calf$ are infinite cyclic and since we are considering the commensurability relation, two edges in $T_\calf$ are in the same cylinder if and only if the intersection of their stabilisers is nontrivial. In particular, this shows that $g$ is trivial and that the action of $G\rtimes_\Phi \ZZ$ on $T_c$ is acylindrical.
\end{proof}

\begin{lemma}\label{Lem:sporadiccase}
    Let $(G,\calf)$ be a sporadic free product of groups and let $\Phi \in \Aut(G,\calf)$. If for each $[A]\in \calf$, the group $A_\Phi$ is in $\FJCX$, then $G \rtimes_\Phi \ZZ$ is in $\FJCX$.
\end{lemma}

\begin{proof}
    Let $T_\calf$ be the Bass--Serre tree associated with $\calf$ and let $T_c$ be its tree of cylinders relative to the commensurability relation. We want to apply Theorem~\ref{Thm:FJCacylindrical} to the action of $G\rtimes_\Phi \ZZ$ on $T_c$. This action is acylindrical by Lemma~\ref{Lem:Actiontreecylacylindrical}. Thus, it suffices to prove that every vertex stabiliser belongs to $\FJCX$. 
    
    Recall that we have a bipartition of $VT_c=V_0T_c \coprod V_1V_c$, where vertices in $V_0T_c$ correspond to vertices of $T_\calf$ and vertices in $V_1T_c$ correspond to cylinders of $T_c$. 
    
    If $v \in V_0T_c$, then its stabiliser is isomorphic to $A_\Phi$.  In that case, the stabiliser of $v$ belongs to $\FJCX$ by hypothesis. 

    Suppose now that $v\in V_1 T_c$. By Lemma~\ref{Lem:Stabcylinder}, the stabiliser of $v$ is isomorphic to $\mathrm{Per}(\Phi) \rtimes_\Phi \ZZ$. By \Cref{Lem:PerFJC}, the group $\mathrm{Per}(\Phi) \rtimes_\Phi \ZZ$ belongs to $\FJCX$.

    Thus, every vertex stabiliser of the action of $G\rtimes_\Phi \ZZ$ on $T_c$ belongs to $\FJCX$. By Theorem~\ref{Thm:FJCacylindrical}, the group $G\rtimes_\Phi \ZZ$ belongs to $\FJCX$.
\end{proof}

\subsection{End of the proof of Theorem~\ref{Thm:combinationthm} }

Let $G=G_1 \ast \ldots \ast G_k \ast F_N$ be a free product of groups, let $\calf'=\{[G_1],\ldots,[G_k]\}$ and let $ \Phi \in \Aut(G,\calf')$. We prove by induction on $k+N$ that $G\in \FJCX$. 

Suppose first that $k+N=1$. If $N=0$, then $G=G_1$ and $G \rtimes_\Phi \ZZ \in \FJCX$ by hypothesis. If $k=0$, then $G=\ZZ$, $G \rtimes_\Phi \ZZ$ is solvable and the result follows from \cite{wegner2015farrell}. This proves the base case.

Suppose now that $k+N \geq 2$ and let $\calf$ be a maximal $\Phi$-periodic free factor system. We may assume, up to taking a power of $\Phi$, that $\calf$ is $\Phi$-invariant, so that we can view $\Phi$ as an element of $\Aut(G,\calf)$. This is possible by Theorem~\ref{Thm:FJCclosedsubgroups}~$(2)$ as, for every $n\in \NN$, the group $G \rtimes_{\Phi^n} \ZZ$ is a finite index subgroup of $G \rtimes_\Phi \ZZ$. 

By induction hypothesis, for every $[A]\in \calf$, the group $A_\Phi$ belongs to $\FJCX$. Combining the nonsporadic case (Lemma~\ref{Lem:fullyirredcase}) and the sporadic case (Lemma~\ref{Lem:sporadiccase}), we conclude that $G \rtimes_\Phi \ZZ$ belongs to $\FJCX$. This concludes the proof.
\hfill\qedsymbol

\subsection{Proving \texorpdfstring{\Cref{thmx:main}}{the main theorem}}

We first record a corollary of \Cref{Thm:combinationthm}.

\begin{corollary}\label{cor:NRH}
    Let $(G,\calp)$ be a virtually torsion-free relatively hyperbolic group with $\calp$ finite and let $\Phi\in\Aut(G,\calp)$.  If for every $[P]\in \calp$ we have $P_\Phi \in \FJCX$, then $G_\Phi \in \FJCX$. 
\end{corollary}

\begin{proof}
By \Cref{Thm:FJCclosedsubgroups} we may assume $G$ is torsion-free.  Let $\calf$ be the minimal free factor system of $G$ such that, for every $[P] \in \calp$, there exists $[A] \in \calf$ with $P \subseteq A$. Since $\Phi \in \Aut(G,\calp)$, by minimality of $\calf$, we have $\Phi \in \Aut(G,\calf)$. Let $[A] \in \calf$. We denote by $\calp_A$ the peripheral structure of $A$ induced by $\calp$. Since $G$ is torsion-free, the group $A$ is one-ended hyperbolic relative to $\calp_A$. By \Cref{Cor:OneEndedFJCrelhyp} the group $A_\Phi$ belongs to $\FJCX$. By \Cref{Thm:combinationthm}, the group $G_{\Phi}$ belongs to $\FJCX$.
\end{proof}

Finally, combining \Cref{Cor:OneEndedFJCrelhyp} and \Cref{cor:NRH} proves our first theorem from the introduction.

\setcounter{thmx}{0}
\begin{thmx}\label{thmx:main}
    Let $(G,\calp)$ be a virtually torsion-free or one-ended relatively hyperbolic group with $\calp$ finite and let $\Phi\in\Aut(G,\calp)$.  If for every $[P]\in \calp$ we have $P_\Phi \in \FJCX$, then $G_\Phi \in \FJCX$. 
\end{thmx}

We now discuss the (minor) changes to the proof used to prove the following theorem.

\setcounter{thmx}{4}        
\begin{thmx} \label{thmx:acvnil}
    Suppose $(G, \calp)$ is one-ended or virtually torsion free, and hyperbolic relative to finitely many conjugacy classes of virtually polycyclic subgroups. Then for every automorphism $\Phi$ of $G$, $\Gamma \coloneqq G \rtimes_{\Phi} \Z$ is in $\AC(\VNil)$.
\end{thmx}

\begin{proof}
    Knopf's work on acylindrical actions of trees applies equally well in the setting of $\AC(\VNil)$ (see \cite[Corollary~4.2]{knopf2019acylindrical} and \cite[Theorem~2.4]{bestvina2023farrell}, note that Knopf does not state this but it is implicit in her work).  We use the same trees as every step of the proof of \cref{thmx:main}. Whenever a vertex group is identified as $\Per(\Phi) \rtimes_\Phi \ZZ$, use \cref{Coro:Slenderperiodicisfixedoneended} to further identify it as some $\Fix(\Phi^k) \rtimes_\Phi \ZZ$. As this has a finite index subgroup isomorphic to $\Fix(\Phi^k) \times \ZZ$, and $\AC(\VNil)$ passes to both subgroups and direct products, this vertex group lies in $\AC(\VNil)$.
\end{proof}

\begin{remark}
    Note that in both \Cref{thmx:main} and \Cref{thmx:acvnil} we may replace the hypothesis `virtually torsion free' with the more general assumption that $G$ has a finite index subgroup whose Dunwoody decomposition has no non-trivial edge groups.  We do not pursue this here.
\end{remark}

\section{Proofs of the applications}\label{sec:applications}
Our first application is to extensions of groups with relatively hyperbolic kernel.

\setcounter{thmx}{1}
\begin{corx}\label{Corx:Extensions}
    Let $(N,\calp)$ be a virtually torsion-free or one-ended relatively hyperbolic group such that $\calp$ consists of finitely many conjugacy classes of groups which are NRH and whose suspensions $P \rtimes_\Psi \ZZ$ are in $\FJCX$ for all automorphisms $\Psi$ of $P$.  Let $1\to N\to \Gamma\to Q\to 1 $ be a short exact sequence.  If $Q$ is in $\FJCX$, then $\Gamma$ is in $\FJCX$.
\end{corx}
\begin{proof}
    Since for all $[P]\in\calp$ the group $P$ is NRH we have that $\Aut(G,\calp)$ is a finite index subgroup of $\Aut(G)$.  Let $\Phi\in\Aut(G)$.  The suspension $G_\Phi$ has a finite index subgroup $G_{\Phi^n}$ such that $\Phi^n\in\Aut(G,\calp)$.  Now, \Cref{thmx:main} implies that $G_{\Phi^n}$ is in $\FJCX$.  It follows from \Cref{Thm:FJCclosedsubgroups} that $G_\Phi$ is in $\FJCX$.  The result now follows from \Cref{thm:extensions}.
\end{proof}

Our other application is that $\Aut(G)$ is in $\FJCX$ for $G$ a one-ended group hyperbolic relative to finitely many conjugacy classes of polycyclic subgroups.  Before we prove this, we collect some results.

\begin{thm}[{\cite[Theorem~4.3]{GuirardelLevitt2015}}]
\label{thm:GLses}
    Let $(G,\calp)$ be a relatively hyperbolic group.  Suppose for every $[P]\in\calp$, the group $P$ is finitely generated.  If $G$ is one-ended relative to $\calp$, then there is a short exact sequence
    \[
        1 \to \mathfrak T \to  \Out_0(G,\calp) \to \prod_{i=1}^p \MCG_0(S_i) \times \prod_j \Out(P_j,\mathrm{Inc}_{P_j}^{(t)}) \to 1
    \]
    where
    \begin{enumerate}
        \item $\mathfrak T$ is a quotient of a finite direct product where each factor is virtually cyclic or contained in some $P$ for $P\in\calp$;
        \item $\MCG_0(S_i)$ maps onto a finite index subgroup of the extended mapping class group $\MCG^\ast(S_i)$ with finite kernel (they are virtually isomorphic).
    \end{enumerate}
\end{thm}

\begin{prop}\label{prop:OutAutpolycyclic}
    If $G$ is a virtually polycyclic group, then $G$, $\Out(G)$ and $\Aut(G)$ are in $\FJCX$.
\end{prop}
\begin{proof}
    By \cite[Theorem~1.1]{BauesGrunewald2006} we see that $\Out(G)$ is an arithmetic group.  Hence, $\Out(G)$ is in $\FJCX$ by \cite{BartelsFarrellLuck2014}.  Technically they only prove the conjecture for $K$- and $L$-theory but it follows for $A$-theory by \cite{Ruping2-16}, \cite[Proof of Theorem~1.8(a)]{knopf2019acylindrical}, and \cite[Theorem~6.19]{EnkelmannLuckPieperUllmannWinges}.  Alternatively, one may use \cite{EnkelmannLuckPieperUllmannWinges} and \cite{KaprowskiUllmannWegnerWinges2018}.  
    
    Since $G$ is virtually soluble it is in $\FJCX$ by \cite{wegner2015farrell} (for $K$- and $L$-theory), \cite{KaprowskiUllmannWegnerWinges2018} (for $A$-theory), and \Cref{Thm:FJCclosedsubgroups}.  Now, a virtually polycyclic group is poly-\{virtually cyclic\}, so any extension $G\rtimes\ZZ$ is also virtually polycyclic.  Thus, $G\rtimes\ZZ$ is in $\FJCX$.  Further, note that $G/Z(G)$ is virtually polycyclic and so in $\FJCX$.  Combining these observations with \Cref{thm:extensions} shows that $\Aut(G)$ is in $\FJCX$.
\end{proof}

\begin{prop}
\label{prop:MCG(orb)}
    The mapping class group of a hyperbolic 2-orbifold is in $\FJCX$.
\end{prop}
\begin{proof}
    This follows from the result for (orientable) surfaces \cite{BartelsBestvina2019Farrell} and assembling results in the literature.  Note that Bartels--Bestvina only prove FJC for $K$- and $L$-theory but the result for $A$-theory follows (as usual) from \cite[Proof of Theorem~1.8(a)]{knopf2019acylindrical}, and \cite[Theorem~6.19]{EnkelmannLuckPieperUllmannWinges}.
    
    Let $S$ be a hyperbolic orbifold, and let $\Sigma$ be an orientable surface covering $S$ with finite degree so that $\pi_1(\Sigma)$ is characteristic in $\pi_1(S)$. (This can be achieved by taking any covering surface, passing to its orientation cover if necessary, and then taking the characteristic core of the corresponding subgroup and realising the covering surface.) By \cite{KolbeEvans2021Orbifold} there is an injective map from $\Aut_{\text{geom}}(\pi_1(S))$ to $\Aut_{\text{geom}}(\pi_1(\Sigma))$, where these \emph{geometric} automorphism groups are exactly the lifts of the mapping class groups.
    
    Restricting to the image, $\Inn(\pi_1(S))$ will be normal, and by the third isomorphism theorem the quotient is isomorphic to $C/(\Inn(\pi_1(S))/\Inn(\pi_1(\Sigma)))$, where $C$ is a subgroup of $\Aut_{\text{geom}}(\pi_1(\Sigma))/\Inn(\pi_1(\Sigma))$, the mapping class group of $\Sigma$. The quotient of inner automorphism groups is finite (in fact isomorphic to the deck transformations $\pi_1(S)/\pi_1(\Sigma)$), so we have realised $\MCG(S)$ as an extension \[1 \to F \to C \to MCG(S) \to 1.\]
    Since mapping class groups of surfaces are residually finite by \cite{Grossman1974MCGrf} (and residual finiteness passes to subgroups), we may apply \cref{lem:finite_extensions} to obtain the conclusion.
\end{proof}

\begin{prop}
    \label{prop:out_in_fjc}
    If $G$ is a one-ended group hyperbolic relative to finitely many conjugacy classes of virtually polycyclic groups, then $\Out(G)$ is in $\FJCX$.
\end{prop}

\begin{proof}
    By \Cref{thm:GLses} there is a finite index subgroup $\Out_0(G)$ fitting into a short exact sequence \[ 1 \to \mathfrak T \to \Out_0(G) \to \prod_{i=1}^p \MCG_0(S_i) \times \prod_j \Out(P_j,\mathrm{Inc}_{P_j}^{(t)}) \to 1.\]

    We want to apply \cref{thm:extensions} to this short exact sequence. First we check the kernel $\mathfrak{T}$: this is a quotient of a direct product of virtually polycyclic groups, and hence is itself virtually polycyclic, and so in $\FJCX$ by \Cref{prop:OutAutpolycyclic}.
    
    Now consider the image. The subgroups $\Out(P_j,\mathrm{Inc}_{P_j}^{(t)})$ are subgroups of $\Out(P_j)$ for a virtually polycyclic $P_j$, and hence are in $\FJCX$ by \cref{prop:OutAutpolycyclic}. Each $\MCG(S_i)$ maps with finite kernel onto the mapping class group of a hyperbolic 2-orbifold. By \cref{lem:finite_extensions} it is enough to consider these mapping class groups. These are in $\FJCX$ by \cref{prop:MCG(orb)}. Then the product is in $\FJCX$ by \cref{Thm:FJCclosedsubgroups}. 

    To apply \cref{thm:extensions} it remains to check the preimages of elements. These are of the form $\mathfrak{T} \rtimes \ZZ$, which are in $\FJCX$ by \Cref{prop:OutAutpolycyclic} since they are again virtually polycyclic. \qedhere

\end{proof}

\begin{thmx}\label{thmx:Aut_FJC}
    If $G$ is a one-ended group hyperbolic relative to finitely many conjugacy classes of virtually polycyclic groups, then $\Aut(G)$ and $\Out(G)$ are in $\FJCX$.
\end{thmx}

\begin{proof}
    Since $\calp$ contains only finitely many conjugacy classes of virtually polycyclic groups, by \cite[Corollary~1.14]{DrutuSapir2005}, we may modify $\calp$ such that that for every $[P]\in\calp$, the group $P$ is NRH.  Observe that since $G$ is hyperbolic relative to finitely many conjugacy classes of virtually polycyclic subgroups its centre is finite.  Hence, $\Inn(G)$ is quasi-isometric to $G$ and again hyperbolic relative to finitely many conjugacy classes of NRH virtually polycyclic subgroups by \cite{BehrstockDrutuMosher2009}. By Proposition~\ref{prop:out_in_fjc}, the group $\Out(G)$ is in $\FJCX$. The result now follows from applying \Cref{Corx:Extensions} to the short exact sequence \[ 1 \to \Inn(G) \to \Aut(G) \to \Out(G) \to 1. \qedhere\]
\end{proof}

\Cref{thmx:Aut_FJC} allows us to prove that the outer automorphism groups of some small complexity relatively hyperbolic groups also belong to $\FJCX$.

\begin{corollary}\label{Cor:AutofreeproductFJC}
    Let $G=A \ast_C B$, where $A$ and $B$ are one-ended hyperbolic groups relative  to finitely many conjugacy classes of virtually polycyclic groups and $C$ is a finite group. The groups $\Out(G)$ and $\Aut(G)$ are in $\FJCX$. 
\end{corollary}

\begin{proof}
    We prove the result for $\Out(G)$, the proof for $\Aut(G)$ being identical to the proof of \Cref{thmx:Aut_FJC} (this uses \Cref{prop:sporadic_stallings_dunwoody} when $C$ is nontrivial and \Cref{thmx:main} otherwise). Let 
    $\Out^0(G)$ be the index (at most) $2$ subgroup of $\Out(G)$ preserving the conjugacy classes of $A$ and $B$. By~\cite{Forester2002Rididity}, every element $\phi \in \Out^0(G)$ has a representative $\Phi\in \phi$ such that $\Phi(A)=A$ and $\Phi(B)=B$. Moreover, the map sending $\phi$ to $\Phi$ defines an isomorphism between $\Out^0(G)$ and $\Aut(A,C) \times \Aut(B,C)$. By \Cref{thmx:Aut_FJC}, the groups $\Aut(A,C)$ and $\Aut(B,C)$ belong to $\FJCX$. By \Cref{Thm:FJCclosedsubgroups}, the groups $\Out^0(G)$ and $\Out(G)$ belong to $\FJCX$.
\end{proof}

\begin{corollary}\label{Cor:AutHNNFJC}
   Let $G=A \ast_C$, where $A$ is a one-ended hyperbolic group relative  to finitely many conjugacy classes of virtually polycyclic groups and $C$ is a finite group. The groups $\Out(G)$ and $\Aut(G)$ are in $\FJCX$.
\end{corollary}

\begin{proof}
    As above we only prove the result for $\Out(G)$. Let $t$ be a stable letter for the HNN extension $A \ast_C$. By~\cite{Levitt2005Rigidity}, the group $\Out(G)$ has an index $2$ subgroup $\Out^0(G)$ such that any element $\phi \in \Out^0(G)$ has a representative $\Phi \in\phi$ such that $\Phi(A)=A$ and $\Phi(t)=ta$ for some $a \in A$. Moreover, the map sending $\phi$ to $\Phi$ induces an isomorphism between $\Out^0(G)$ and $A \rtimes \Aut(A,C)$. Thus, $\Out^0(G)$ fits in a short exact sequence \[1 \to A \to \Out^0(G) \to \Aut(A,C) \to 1.\]
    The group $\Aut(A,C)$ belongs to $\FJCX$ by \Cref{thmx:Aut_FJC}. Moreover, for every infinite cyclic subgroup $Q \in \Aut(A,C)$, the preimage of $Q$ in $\Out^0(G)$ belongs to $\FJCX$ by \Cref{thmx:main}. Thus, by \Cref{thm:extensions}, the groups $\Out^0(G)$ and $\Out(G)$ belong to $\FJCX$.
\end{proof}

One more infinitely ended case is known, since $\Out(F_2) \cong \GL(2,\ZZ)$ is an arithmetic group, and the same extension arguments as above will give the result for $\Aut(F_2)$. Our techniques do not seem to extend to $\Out(F_n)$, which will be necessary to make any further progress.

\section{Torsion in the infinitely ended case}\label{sec:torsion}

The aim in this section is to prove as much of \cref{sec:infinite_ended} as possible without the assumption that $G$ is virtually torsion free.  To this end we prove a generalisation of the sporadic case (of \cref{subsec:sporadic_case}) to graphs of groups with one finitely stabilised edge. 

\begin{prop}
\label{prop:sporadic_stallings_dunwoody}
    Suppose $G \cong A \ast_C B$ or $G \cong A \ast_C$ with $C$ a finite group. Further suppose that $\Phi \in \Aut(G)$ preserves the conjugacy classes of $A$ and $B$, and that the restrictions of $\Phi$ to $A$ is such that $A \rtimes_{\Phi|_A} \ZZ$ belongs to $\FJCX$ (and similarly for $B$, if applicable). Then $G \rtimes_\Phi \ZZ$ belongs to $\FJCX$.
\end{prop}

The proof being largely an elaboration of the arguments in \cref{subsec:sporadic_case}, here we indicate the necessary changes and references.

\begin{proof}
    First, we have to argue that $\Phi$ preserves the action on the Bass--Serre tree $T$ for this splitting. This follows from \cite{Levitt2005Rigidity}, or for the two vertex case already from \cite{Forester2002Rididity}, which give that these one edge splittings are \emph{rigid}: the unique reduced tree in their deformation spaces, together with the hypothesis that $\Phi$ preserves the vertex groups.

    So we may consider the action of $G \rtimes_\Phi \ZZ$ on $T$. Use $t$ to denote the generator of the $\ZZ$ factor. After possibly passing to a finite index subgroup (by taking the square of $\Phi$, if necessary) we may suppose the quotient graphs are the same for both actions, and that $t$ stabilises an edge $e$. Edge stabilisers are virtually cyclic, and admit a map to $\ZZ$ with a conjugate of $t$ is mapped to the generator. For $G_e$, we may take this preimage to be $t$. 

    Note that by work of Wall~\cite[Lemma 4.1]{Wall1967VirtuallyCyclic} virtually cyclic groups act on the line and have a unique maximal finite normal subgroup which is the kernel of this action. Since $G_e$ surjects onto $\ZZ$, in fact this is the unique maximal finite subgroup.

    Following the proof in the free splitting case, we need to take a tree of cylinders to ensure that we have an acylindrical action. Just as in that case, the commensurablilty relation is admissible, and we take the tree of cylinders $T_c$ relative to this relation.

    We need to adapt the proofs of \cref{Lem:Stabcylinder} and \cref{Lem:Actiontreecylacylindrical} to the new situation. In the first case, we assume that $e$ and $w^{-1}t^je=w^{-1}e$ are edges in the same cylinder. Then $\langle t \rangle$ and $\langle w^{-1} t w \rangle$ are finite index subgroups of the respective stabilisers, and again the commensurability relation implies that there are powers $n,m \in \NN$ so that \[t^n=w^{-1}t^mw=w^{-1}\Phi^m(w)t^m.\] As before, we see that $n=m$ and $w$ is periodic.

    To adapt the proof of \cref{Lem:Actiontreecylacylindrical}, make the adjustments of the first paragraph of that proof and then consider the whole intersection $G_v \cap G_{v'}$. This group fixes edges in two distinct cylinders, so is contained inside some $G_e \cap G_{e'}$, where this intersection is between two virtually cyclic subgroups that are not commensurable. In particular, this means the intersection is finite; in fact its cardinality is bounded by the size of the unique maximal finite subgroup (in either -- they are conjugate).  This means the action of $G \rtimes_\Phi \ZZ$ on $T_c$ is acylindrical.

    To finish the proof, we recall the bipartite nature of $T_c$, and observe that vertex stabilisers are either isomorphic to $A \rtimes_{\Phi|_A} \ZZ$ (or the same for $B$ -- the original vertex stabilisers), or a cylinder stabiliser $\Per(\Phi) \rtimes_\Phi \ZZ$. The first kind is in $\FJCX$ by hypothesis; the second by \cref{Lem:PerFJC}.
\end{proof}

With this in hand, one can begin to try and run the induction argument of \cref{sec:infinite_ended} on a Stallings--Dunwoody decomposition of a more general infinitely ended group. However, there seems as yet to be no analogy for the relative hyperbolicity argument used in the non-sporadic case, and so the induction will not be able to proceed if at some stage we encounter a maximal periodic ``Stallings--Dunwoody type splitting'' that has more than one edge, and at least one edge with non-trivial stabiliser.

A proof of the following conjecture, the analogy of \cref{thm:relativehyperbolicityfullyirred} for general infinite ended groups, should complete the proof of \Cref{thmx:main} with no assumption on torsion. 

\begin{conjecture}
    Suppose $G$ is the fundamental group of a non-sporadic graph of groups with finite edge stabilisers, and $\Phi \in \Aut(G)$ is fully irreducible relative to this splitting. Then $G \rtimes_{\Phi^N} \ZZ$ is hyperbolic relative to the suspensions of polynomially growing subgroups of $\Phi$.
\end{conjecture}

It may be necessary to assume accessibility in the previous conjecture but for now we do not.  As in the free product case, the correct notion of growth should be with respect to the translation length function for the action of $G$ on the Bass--Serre tree. The correct definition of \emph{fully irreducible} appears to be that in any splitting (strictly) dominated by ours every power of $\Phi$ does not preserve the set of elliptic subgroups.

\section*{Statements and declarations}

\noindent{\bf Availability of data and material:} \textit{Not applicable}

\noindent{\bf Financial interests:}
\textit{The authors declare they have no financial interests.}

\bibliographystyle{halpha}
\bibliography{refs.bib}

\end{document}